\pdfoutput=1 
\RequirePackage{fix-cm} 
\documentclass[preprint,12pt,leqno]{elsarticle}
\usepackage{fixltx2e} 

\usepackage[dvipsnames,svgnames,x11names,hyperref]{xcolor} 
\usepackage{url}

\usepackage{amsmath,amsthm,amsfonts}
\usepackage{amssymb}
\usepackage{mathtools}
\mathtoolsset{showonlyrefs=false,showmanualtags}
\RequirePackage{suffix}
\RequirePackage{ifmtarg}
\RequirePackage{xifthen}
\RequirePackage{xkeyval}
\RequirePackage{undertilde}
\usepackage{hyperxmp}
\usepackage[		pdfborder={0 0 0},%
								colorlinks=true,%
								linkbordercolor={0 0 0},%
								hyperfootnotes=true,%
								pdftex=true, %
								implicit=true,%
								draft=false,%
								final=true,%
								raiselinks=true,%
								bookmarks=true,                       
                bookmarksnumbered=true,               
								pdfpagemode=UseOutlines,                
	              breaklinks=true,                 
	              pageanchor=true,                 
	              plainpages=false,                
	              pdfpagelabels=true,                   
	              pagebackref=true,                     
	              hyperindex=true,                      
								pdfdisplaydoctitle=true,%
								pdfauthor={Peter Cholak and Peter M. Gerdes and Karen Lange},%
								pdfcreator={LaTeX (pdflatex) using package hyperref},%
								pdfkeywords={r.e.\ sets, c.e.\ sets, automorphisms of E, Q(A),almost prompt,tardy,very-tardy,n-tardy},%
								pdfsubject={MSC 2010 Primary: 03D25, 03D30},%
]{hyperref}

\numberwithin{equation}{section}
\newtheorem{theorem}[equation]{Theorem}
\newtheorem{lemma}[equation]{Lemma}

\newtheorem{question}[equation]{Question}

\newtheorem{corollary}[equation]{Corollary} 
\theoremstyle{definition}
\newtheorem{definition}[equation]{Definition}

\theoremstyle{remark}

\newtheorem{remark}[equation]{Remark}

\newcommand{\curnode}[2]{\alpha(#1,#2)}
\newcommand{\converge}{\!\downarrow}
\newcommand{\diverge}{\!\uparrow}
\newcommand{\cat}{\mathbin{\hat{}}}
\newcommand*{\lh}[1]{\lvert#1\rvert}
\newcommand*{\pair}[2]{\mathopen{\langle} #1, #2 \mathclose{\rangle}}

\DeclareMathOperator \dom{dom}

\newcommand{\reset}[2]{\mathfrak{r}(#1,#2)}
\newcommand{\Lstg}[2]{{#2}_{#1}}
\newcommand{\rstrn}[3][]{r_{#1}(#2,#3)}

\makeatletter

\newcommand*{\conv}[1][]{\mathpunct{\downarrow}_{#1}}
\newcommand*{\@recthy@llangle}{\mathopen{\langle\!\langle}}
\newcommand*{\@recthy@rrangle}{\mathclose{\rangle\!\rangle}}
\newcommand*{\@recthy@EmptyStr}{\lambda}
\newcommand*{\@recthy@concatSYM}{\widehat{}}

\newcommand*{\restr}[1]{\mathpunct{\restriction_{#1}}}

\newcommand*{\map}[2]{:#1\mapsto #2}
\newcommand*{\functo}[3]{\ensuremath{#1\map{#2}{#3}}}

\newcommand*{\bstrs}{2^{<\omega}}

\newcommand*{\cantor}{2^{\omega}}
\WithSuffix\def\cantor*{\left(2\union \diverge \right)^{\omega}}
\newcommand*{\baire}{\omega^{\omega}}
\WithSuffix\def\baire*{\left(\omega \union \diverge \right)^{\omega}}

\newcommand*{\@recthy@TSYM}{\mathbf{T}}
\newcommand*{\Tequiv}{\mathrel{\equiv_{\@recthy@TSYM}}}

\newcommand*{\nTequiv}{\mathrel{\ncong_{\@recthy@TSYM}}}

\newcommand*{\Tlneq}{\lneq_{\@recthy@TSYM}}
\newcommand*{\Tleq}{\leq_{\@recthy@TSYM}}
\newcommand*{\Tgneq}{\gneq_{\@recthy@TSYM}}
\newcommand*{\Tgeq}{\geq_{\@recthy@TSYM}}
\newcommand*{\Tgtr}{>_{\@recthy@TSYM}}
\newcommand*{\Tless}{<_{\@recthy@TSYM}}
\newcommand*{\nTleq}{\nleq_{\@recthy@TSYM}}
\newcommand*{\nTgeq}{\ngeq_{\@recthy@TSYM}}

\newcommand*{\Tdegjoin}{\mathbin{\vee_{\@recthy@TSYM}}}
\newcommand*{\Tdegmeet}{\mathbin{\wedge_{\@recthy@TSYM}}}

\newcommand*{\TsetJoin}{\mathbin{\oplus}}
\WithSuffix\def\TsetJoin[#1][#2]{\mathop{\bigoplus}^{#2}_{#1}}

\newcommand*{\Tdeg}[1]{\utilde{#1}}

\let\jump=\Tjump

\newcommand*{\jjump}[1]{#1''}
\newcommand*{\Tzerosym}{0}

\newcommand*{\zerojj}{\jump{\jump{\Tdeg{\Tzerosym}}}}

\newcommand*{\REset}{\let\@PMG@parenarg\@PMG@undefined\let\@PMG@braketarg\@PMG@undefined\@REsetbody}
\newcommand*{\@REsetbody}[1]{W\ifdefined\@PMG@parenarg
^{\@PMG@parenarg}%
\fi%
\ifdefined\@PMG@braketarg
_{#1, {\@PMG@braketarg}}%
\else
_{#1}
\fi}
\WithSuffix\def\@REsetbody(#1){\def\@PMG@parenarg{#1}\@REsetbody}
\WithSuffix\def\@REsetbody[#1]{\def\@PMG@braketarg{#1}\@REsetbody}
\WithSuffix\def\exists[#1]{\left(\exists\, #1 \right)\!}
\WithSuffix\def\forall[#1]{\left(\forall\, #1 \right)\!}
\newcommand*{\str}[1]{\mathopen{\@recthy@llangle}#1\mathclose{\@recthy@rrangle}}

\newcommand*{\EmptyStr}{\@recthy@EmptyStr}
\newcommand*{\estr}{\EmptyStr}

\newcommand*{\concat}{\mathbin{\@recthy@concatSYM}}
\WithSuffix\def\concat[#1]{\concat\relax#1}

\newcommand*{\set}[2]{\ifthenelse{\isempty{#2}}{\left \{  #1 \right \}}{\left \{ #1 \middle | #2\right \}} }
\newcommand*{\card}[1]{\lvert#1\rvert}
\newcommand*{\union}{\mathbin{\cup}}

\newcommand*{\isect}{\mathbin{\cap}}

\newcommand*{\eset}{\emptyset}
\newcommand*{\nin}{\mathrel{\not\in}}
\newcommand*{\recfnlSYM}{\Phi}
\newcommand*{\murec}[2]{\ensuremath{\mu#1\left( #2 \right) }}

\newcommand*{\recfnl}[4][]{\recfnlSYM_{#2\ifthenelse{\isempty{#1}}{}{,#1} }%
	\ifthenelse{\isempty{#4}}%
	{\ifthenelse{\isempty{#3}}%
		{}
		{(#3)}
	}
	{\ifthenelse{\isempty{#3}}{(#4)}{(#3;#4)}%
}}

\makeatother
\journal{Annals of Pure and Applied Logic}

\begin{document}

\begin{frontmatter}

	\title{On $n$-Tardy Sets}

	\author[und]{Peter~A.~Cholak\corref{cor1}\fnref{fn1}}


	\ead{Peter.Cholak.1@nd.edu}
	\ead[url]{http://www.nd.edu/~cholak}


	\cortext[cor1]{Corresponding author} 

	\fntext[fn1]{Cholak was partially supported by NSF DMS-0652669 and  NSF-DMS-0800198.}

	\author[und]{Peter M. Gerdes\fnref{fn2}}


	\ead{gerdes@invariant.org}
	\ead[url]{http://invariant.org}
	\fntext[fn2]{Gerdes was partially supported by NSF EMSW21-RTG-0739007 and EMSW21-RTG-0838506.}


	\author[und]{Karen Lange\fnref{fn3}}


	\ead{klange1@nd.edu}
	\ead[url]{http://www.nd.edu/~klange1/}


	\fntext[fn3]{Lange was partially supported by NSF DMS-0802961.}

	\address[und]{Department of Mathematics\\ University of Notre Dame\\ 
	  Notre Dame, IN 46556-5683}

\begin{abstract} 
  Harrington and Soare introduced the notion of an $n$-tardy set.
  They showed that there is a nonempty $\mathcal{E}$ property $Q(A)$
  such that if $Q(A)$ then $A$ is $2$-tardy.  Since they also showed
  no $2$-tardy set is complete, Harrington and Soare showed that there
  exists an orbit of computably enumerable sets such that every set in that orbit is
  incomplete.  Our study of $n$-tardy sets takes off from where
  Harrington and Soare left off. We answer all the open questions
  asked by Harrington and Soare about $n$-tardy sets.  We show there
  is a $3$-tardy set $A$ that is not computed by any $2$-tardy set
  $B$. We also show that there are nonempty $\mathcal{E}$ properties
  $Q_n(A)$ such that if $Q_n(A)$ then $A$ is properly $n$-tardy.
\end{abstract}


\end{frontmatter}

\section{Introduction}

Let $\mathcal{E}$ denote the structure of c.e.\ sets under the language of inclusion.  Understanding the interplay between computability and  definability in $\mathcal{E}$ is  a longstanding area of research in classical computability theory.  In 1944 \cite{Post:RE-sets-and-deciscion-problems}, Post set out to find an incomplete noncomputable c.e.\ set, i.e., a noncomputable c.e.\ set that does not have the degree of the halting problem $K$.   He defined several properties of c.e.\ sets (such as {\em simplicity}) in the hope that no c.e.\ set satisfying one of these properties could be complete. All of the properties he suggested failed to satisfy this condition, but many of them  are definable in $\mathcal{E}$.   Although Friedberg and Muchnik \cite{Mucnik:56:incomplete-degree,Friedberg:57:incomplete-degree}  famously obtained an incomplete noncomputable c.e.\ set using a priority argument, a natural question is whether  there exists an $\mathcal{E}$-definable nontrivial property $Q$ such that if $Q(A)$ holds, then $A$ is an incomplete noncomputable c.e.\ set.  Harrington and Soare produced such a property $Q$ in \cite{posts-program-and-incomplete-recursively-enumerable-sets}, and they also described an $\mathcal{E}$-definable property that guarantees completeness (See \cite{SoareBook}, p. 339 and \cite{definability-automorphisms-and-dynamic-properties-of-computably-enumerable}).  These results are part of work by many towards the following general goal.

\begin{question}\label{Q:AuttoComp}
Characterize what sets are and are not automorphic to a complete set.  
\end{question}

Harrington and Soare showed that all sets that satisfy $Q$ are {\em $2$-tardy} \cite{codable-sets-and-orbits-of-computably-enumerable-sets}, a slowness condition that we describe, along with the conditions {\em $n$-tardy} and {\em very tardy}, in \S \ref{SS:tardy}.  The very tardy sets, by definition, are those that are not almost prompt, and all complete sets are prompt.    All $n$-tardy  sets are  very tardy and, hence,  incomplete.  Thus, any $A$ for which $Q(A)$ holds is not automorphic to a complete set.  On the other hand, Harrington and Soare \cite{the-delta-0-3-automorphism-method}, building on  work of the first author, Downey, and Stob \cite{automorphisms-of-the-lattice-of-recursively-enumerable-sets-promptly}, proved that every almost prompt set (i.e., every not very tardy set) is automorphic to a complete set.  Thus, in order to work towards answering Question \ref{Q:AuttoComp}, we explore the very tardy sets and their orbits from the perspectives of computability and definability.  We begin by defining the almost prompt sets.

\subsection{Almost prompt}\label{SS:prompt}

\begin{definition}
  A set $X$ is $n$-c.e.\ iff there is a computable sequence of c.e.\
  sets $\{X_i:1\leq i \leq n\}$ such that
  \begin{equation*}
    X = (X_1 - X_2) \cup \ldots \cup (X_{n-2} - X_{n-1})\cup X_n \mbox{ if $n$ is odd, and }
         \end{equation*}
      \begin{equation*}
         X = (X_1 - X_2) \cup \ldots \cup (X_{n-1} - X_{n}) \mbox{ if $n$ is even.}
     \end{equation*}
  The sequence of sets $\{X_i:1\leq i \leq n\}$ is  an $n$-c.e.\
  presentation of $X$.  Such a sequence can be used to give a
  stagewise approximation of $X$: (for the case when $n$ is even)
 \begin{equation*}
    X_s = (X_{1,s} - X_{2,s}) \cup \ldots \cup (X_{n-1,s} - X_{n,s}) 
  \end{equation*}
  Such a sequence is denoted $X^n_e$,  and $X^n_{e,s}$ denotes the stagewise approximation
  given by the sequence.
\end{definition}

\begin{definition}[Definition 11.3 of \cite{the-delta-0-3-automorphism-method}]
  Let $A$ be a c.e.\ set and $\{A_s\}$ be an enumeration of $A$.  The set $A$ is
  \emph{almost prompt} iff there is a nondecreasing function $p(s)$
  such that for all $n$ and all $e$
  \begin{equation}\label{eq:1}
    X^n_e = \overline{A} \implies (\exists x ) ( \exists s) [ x 
    \in X^n_{e,s} 
    \wedge x \in A_{p(s)}].
  \end{equation}
\end{definition}

Harrington and Soare \cite{the-delta-0-3-automorphism-method} showed that this definition is robust. That is, if Equation~(\ref{eq:1}) holds for some enumeration of $A$, it holds for all enumerations of $A$ (see \cite[Theorem~11.4]{the-delta-0-3-automorphism-method}). They also proved that any c.e.\ set of prompt degree is almost prompt (see \cite[Theorem~11.7]{the-delta-0-3-automorphism-method}); thus, the notion of almost prompt generalizes the notion of prompt. They also showed that almost prompt sets are ubiquitous in the following sense.

\begin{theorem}[Harrington, Soare,  Theorem~11.12 \cite{the-delta-0-3-automorphism-method}]\label{T:HSalmostprompt}
There are almost prompt sets of every c.e.\ degree.  
\end{theorem}

Moreover, they showed that there are tardy (i.e., not of prompt degree) sets $A$ such that every degree Turing above $A$ is almost prompt, (see \cite[Theorem~11.8]{the-delta-0-3-automorphism-method}) and that the join of an almost prompt set and any computably enumerable set is almost prompt (see \cite[Theorem~11.11]{the-delta-0-3-automorphism-method}).

In order to show Theorem \ref{T:HSalmostprompt}, Harrington and Soare proved that every low simple set is almost prompt (see \cite[Theorem~11.10]{the-delta-0-3-automorphism-method}). They left the following question open:

\begin{question}[Question 1 of \cite{the-delta-0-3-automorphism-method}]\label{Q:Q1HS}
  If $A$ is low$_2$ and simple, is $A$ almost prompt?
\end{question}

We provide a negative answer to Question \ref{Q:Q1HS} in \S \ref{S:low2simple},  but we first focus on particular classes of sets that are not almost prompt.

\subsection{Very tardy and \texorpdfstring{$n$}{n}-tardy sets}\label{SS:tardy}

A degree is {\em tardy} if it is not a prompt degree.  A set is {\em very tardy} if it is not almost prompt.  (Note that being very tardy is a property of sets and does not readily extend to degrees.)  Since the definition of almost prompt is robust, we have the following equivalent definition.  

\begin{definition}\label{D:verytardy}
  Let $A$ be c.e.\ and $\{A_s\}$ be an enumeration of $A$. The set $A$ is
  \emph{very tardy} iff $A$ is not almost prompt iff  for every
  nondecreasing computable function $p(s)$ there is an $n$ and an $e$ such that
  \begin{equation}\label{eq:2}
    X^n_e = \overline{A} \implies (\forall x ) ( \forall s) [ x \in
    X^n_{e,s} 
    \implies x \not\in A_{p(s)}].
  \end{equation}
  Moreover, $A$ is \emph{$n$-tardy} iff there is a  single $n$ that works for all such
  functions $p(s)$, and  $A$ is {\em properly $n$-tardy} if $A$ is $n$-tardy but not $n-1$-tardy.  
\end{definition}

As described in the introduction, Harrington and Soare proved the following theorem.  
\begin{theorem}[Harrington and Soare \cite{posts-program-and-incomplete-recursively-enumerable-sets}]
There exists an $\mathcal{E}$-definable nontrivial property $Q$ such that if $Q(A)$ holds, then $A$ is not automorphic to a complete set. 
\end{theorem}
\noindent
More specifically, they showed the property $Q$  describes a subset of the $2$-tardy sets.  We need a few definitions in order to define this subset.

\begin{definition}\label{D:major}
\begin{enumerate}
\item Let $A\subset_\infty C$ denote that $A\subset C$ and $C-A$ is infinite.
\item A subset $A$ is a {\em major} subset of $C$ if $A\subset_\infty C$ and for all $e$, 
$$\bar{C}\subseteq W_e \implies \bar{A}\subseteq^* W_e.$$

\item A subset $A\subset C$ is a {\em small} subset of $C$ (written $A\subset_s C)$) if $A\subset_\infty C$ and for all $X$ and $Y$, 
$$X\cap(C-A)\subseteq Y\ \implies\
(\exists Z)_{Z\subseteq X}[Z\supseteq(X-C)\ \&\ (Z\cap C)\subseteq Y].$$

\item If $A$ is both a small subset and a major subset of $C$, we call $A$ a {\em small major} subset of $C$ and write $A\subset_{sm} C$.  
\end{enumerate}
\end{definition}

\begin{theorem}[Harrington and Soare \cite{codable-sets-and-orbits-of-computably-enumerable-sets}]
$Q(A) \iff$

\noindent
$ (\exists C)[A\subset_{sm} C\ \&\ $A$ \mbox{ is 2-tardy}].$

\end{theorem}
\noindent
Harrington and Soare used this characterization to show that any $A$ satisfying $Q(A)$ is not automorphic to a complete set.
\begin{definition}
The {\em orbit} of $A$, denoted by $[A]$, is the set of c.e.\ sets $B$ such that there exists an automorphism $\Psi$ of $\mathcal{E}$ sending $A$ to $B$.  
\end{definition}
\noindent
If $A$ satisfies $Q(A)$ and there is an automorphism $\Psi$ of $\mathcal{E}$, then $Q(\Psi(A))$ holds as well.  In other words, $Q$ holds of any element in $[A]$.  Since $Q$ holds of all sets in $[A]$, all sets in $[A]$ are 2-tardy and therefore incomplete.  Thus, if $A$ satisfies $Q(A)$, $A$ is not automorphic to a complete set.  

In \S \ref{S:Q_n}, we define nontrivial  properties $\hat{Q}_n$ that generalize $Q$.  In Theorem \ref{T:Q_2nimptardy}, we show that if $\hat{Q}_{n}(A)$ holds, then $A$ is $n$-tardy and $\neg\hat{Q}_i(A)$ holds for all $i<n$.  In Theorem \ref{T:Q_nsatisfied}, we show that  there is some properly $n$-tardy set $A_n$ for which  $\hat{Q}_n(A_n)$ holds.   Thus, the collection $\{[A_n]\}_{n\in\omega}$ witness that the c.e.\ sets that are not automorphic to a complete set break into  countably many disjoint orbits.

\subsection{Codeable sets}

In \cite{codable-sets-and-orbits-of-computably-enumerable-sets}, Harrington and Soare also explore the connection between 
tardiness and what sets $X$ are coded in every nontrivial orbit in the following sense.  

\begin{definition}[\cite{codable-sets-and-orbits-of-computably-enumerable-sets} Definition 1.3]
\begin{enumerate}
\item We say $X$ is {\em coded in the orbit of $A$}, denoted $X\le_T[A]$, if $X\le_T B$ for some $B\in[A]$.  
\item We say $X$ is {\em codeable} if for every noncomputable set $A$, $X\le_T[A]$.  
\end{enumerate}
\end{definition}

Harrington and Soare obtain the following characterization of  the codeable sets by using the $\Delta_3^0$-automorphism method they developed in \cite{definability-automorphisms-and-dynamic-properties-of-computably-enumerable}.  

\begin{theorem}[Harrington Soare \cite{codable-sets-and-orbits-of-computably-enumerable-sets} Corollary 1.8]\label{T:codeable}
A set is codeable iff $X\le_T D$ for some $D$ satisfying $Q(D)$.  

\end{theorem}

Using Theorem \ref{T:codeable}, Harrington and Soare obtain the following simple corollary.

\begin{corollary}[Harrington Soare \cite{codable-sets-and-orbits-of-computably-enumerable-sets} Corollary 1.9]
If $S$ has prompt degree, then $S$ is not codeable.  
\end{corollary}
\noindent
Harrington and Soare in fact show that a set is codeable iff $X\le_T D$ for some $2$-tardy $D$; Theorem \ref{T:codeable} only uses the fact that if $Q(D)$ holds, then $D$ is 2-tardy.  Thus, the ability to code in the above sense is more connected to enumeration speed than degree-theoretic content.  Given this observation, it is natural to wonder whether all very tardy sets are codeable.  Harrington and Soare asked a more specific version of this problem:

\begin{question}[Harrington Soare \cite{the-delta-0-3-automorphism-method} Question~1]
  Are all $3$-tardy sets codeable?  
\end{question}

By Theorem \ref{T:codeable}, this is is equivalent to the following question.  

\begin{question}\label{Q:3tardy}
  If $A$ is $3$-tardy, does there exist a  $2$-tardy set $B$ such that $A
  \leq_T B$?
\end{question}

Let $A$ be $2$-tardy.  If $A_0 \sqcup A_1 = A$ is a nontrivial split,
then each of the $A_i$ are $3$-tardy.  Given $p(s)$ and $X^2_e$ such that 
Equation~(\ref{eq:2}) holds for $A$, then $p(s)$ and $X^3_{\tilde{e}}
= (X^2_{e_1} - X^2_{e_2}) \sqcup A_{\bar{i}}$ witnesses
Equation~(\ref{eq:2}) holds for $A_i$, where $\bar{0} = 1$ and
$\bar{1}=0$.  Prior to this work, it was unknown whether every $3$-tardy is
the split of a $2$-tardy.  If this was the case, then clearly every $3$-tardy would be computable from  the  $2$-tardy of which it is a split, and hence would be codeable.  In \S \ref{S:3tardy}, we show that not all $3$-tardy sets are splits of $2$-tardy sets.  In fact, we answer Question \ref{Q:3tardy} negatively.  We show that there exists a $3$-tardy set that is not computed by any $2$-tardy.  Hence, not all $3$-tardy (and very tardy) sets are codeable.

\section{A \texorpdfstring{$3$}{3}-tardy not computed by any \texorpdfstring{$2$}{2}-tardy}\label{S:3tardy}

We devote this section to constructing a $3$-tardy set $A$ not computed by any $2$-tardy set.  Hence, $A$ is non-codeable.

\begin{theorem}\label{T:3tardy}
  There is a $3$-tardy set $A$ such that for all $2$-tardy sets $B$,
  $A \nleq_T B$.
\end{theorem}

\begin{proof}
  We will construct $A$. Our construction style will be a pinball
  machine laid out on top of a tree.  Here our tree will be $3^{<
    \omega}$. Since balls move downward (gravity) in this case we want
  to think of our tree as growing upward. As always, we are most
  concerned about the action of the pinball machine along the true
  path. We will have an approximation $f_s$ of the true path $f$, such
  that $f = \liminf_s f_s$.

  The approximation to the true path, $f_s$, will help determine the
  movement of the balls (integers) on the pinball machine.  Balls will
  be placed on the pinball machine by a node $\alpha \subset f_s$ at
  stage $s$ only when we wish to put them into $A$.  At stage $s$, all
  balls $x$ on the machine will be located at some node $\curnode{x}{s}$.
  If $\curnode{x}{s} =\lambda$ ($\lambda$ is the empty node), we put $x$
  into $A_{s+1}$ and remove $x$ from the machine at stage $s+1$.  So
  when a ball $x$ is on the machine our apparent goal is to move $x$
  downward and into $A$. At some point later, we will \emph{sometimes}
  change our mind and remove balls from the machine, preventing them from going
  into $A$. If $f_{s+1} <_{L} \curnode{x}{s}$ we will also remove $x$
  from the machine at stage $s+1$ and never use it again.  At stage
  $s+1$, we are free to place any ball $x\le s+1$ that has never been used on the machine.  However, we must ensure that for all
  $s$, if $\curnode{x}{s}\converge$ then, for all $t$ such that $x \leq t
  \leq s$, $\curnode{x}{s} \leq_{L} f_t$.  The action to ensure this goes
  on at every stage in the background.

  Our next goal is to make $A$ a $3$-tardy set.  This means that balls must
  enter $A$ very slowly.  We have to meet the following requirements:
  \begin{equation*}\tag*{$\mathcal{N}_e$:}
    \label{eq:3}
     \varphi_e \text{ total} \implies \exists[X^3_e]\left(X^3_e = \overline{A}  \land \forall[x]\forall[s]\left[ x \in X^3_{e,s} \implies x \not\in A_{\varphi_e(s)} \right]   \right) 
  \end{equation*}
  In general, the way to meet $\mathcal{N}_e$ is to ensure that for all balls \( x \) there is a stage $s_1$ at which we put $x$ into $A_{s_1}$ or
  $X^3_{e_1,s_1}$.  Now if a ball $x$ in $X^3_{e_1,s_1}$ wants to enter
  $A$ at stage $s_2 > s_1$ we must put $x$ into $X^3_{e_2,s_2}$.  Then
  we wait until a stage $s_3$ such that
  $\varphi_{e,s_3}(s_2)\converge$.  If such a stage $s_3$ exists then
  we \emph{must} eventually put $x$ into $A$ or $X^3_{e_3}$.  If a ball $x$ is in $X^3_{e_2,s_2}$ and we
  remove it from the machine at stage $s_4$, we will put $x$ into
  $X^3_{e_3}$ at stage $s_4$.  If $\varphi_{e,s_3}(s_2)\diverge$ then
  $\varphi_e$ is not total and the requirement is satisfied.

  In the tree construction, we will use node $\gamma$ to meet
  $\mathcal{N}_e$.  We will label the $3$-c.e.\ set constructed at
  $\gamma$, as $X^3_\gamma$ rather than $X^3_e$.  At stages $s$, where
  $\gamma \subseteq f_s$ we will put all balls $x \not\in A_s$ such
  that $|\gamma|\leq x \leq s$ into $X^3_{\gamma_1,s}$.  If $\gamma
  \subset f$ then almost all balls not in $A$ are in $X^3_{\gamma_1}$.
  At the first stage where $\curnode{x}{s} = \gamma$, we will put $x$ into
  $X^3_{\gamma_2}$.  If we remove $x$ from the machine before entering
  $A$, we will put $x$ into $X^3_{\gamma_3}$. Should \( f_s \) ever be to the left of \( \gamma \), then some ball \( x \) with \( \curnode{x}{s} \subseteq f_s \) already in \( X^3_{\gamma_1,s} \) might enter \( A \) without proper delay.  However, since only finitely many such stages may occur along the true path whenever \( f_s \) moves to the left of \( \gamma \), we may reset our construction of \( X^{3}_\gamma \) (equivalently, we imagine that the tree guesses at how many elements each positive requirement places into \( A \)).

  Given a stage $s+1$ such that $\gamma \subseteq f_{s+1}$, let $t
  \leq s$ be the greatest stage such that $\gamma \subseteq f_t$ (if $t$
  does not exist let $t=0$.) Define $$l_\gamma(s) = \max x [(\forall z
  < x)\varphi_{e,s}(z)\converge].$$ 
\noindent
The function $l_\gamma(s)$ measures the length of convergence of $\varphi_e$ at stage $s$.    
  If $l_\gamma(s) > l_\gamma(t)$
  and, for all $x \in X^3_{\gamma_2,s}$, if $x \in
  X^3_{\gamma_2,\text{at } s'}$ then $l_\gamma(s) > s'$,  then we say that $s+1$
  is {\em $\gamma$-expansionary}.    In other words,  stage $s+1$ is $\gamma$-expansionary if the length of convergence of $\varphi_e$ has increased and the proper amount of delay for all $x\in X^3_{\gamma_s, s}$ has been determined.  At $\gamma$-expansionary stages $s+1$, we
  move all balls $x$ such that $\curnode{x}{s} =\gamma$ downward so that
  $\curnode{x}{s+1} = \beta$, where $\beta \cat 0 \subseteq \gamma$ and
  $\beta$ is the greatest such subnode of $\gamma$ assigned to some
  $\mathcal{N}_{e'}$  (only nodes working on the requirements
  $\mathcal{N}_{e'}$ stop balls from moving downwards) or if no such
  $\beta$ exists let $\beta = \lambda$. If $s+1$ is
  $\gamma$-expansionary, we let $\gamma \cat 0 \subset
  f_{s+1}$. Otherwise, we let $\gamma \cat 1 \subset f_s$.  If we have
  moved any balls downwards or $|\gamma| = s$, we end this stage.
  Otherwise, we consider the action of $\gamma \cat 0$ or $\gamma \cat
  1$.

  $\mathcal{N}_e$ is a $\Pi^0_2$ requirement.  How $\mathcal{N}_e$ is
  met depends on the answer to the $\Pi^0_2$ question is $\varphi_e$
  total. Suppose that $\gamma \subset f$. Define $f$ such that
  $\gamma \cat 0 \subset f$ if $\varphi_e$ is total and $\gamma \cat 1
  \subset f$ if not.  If $\gamma \subset f$ then it not hard to that
  $\liminf f_s \restriction (|\gamma|+1) = \gamma \cat 0$ iff
  $\varphi_e$ is total.

	Note that we have made the simplifying assumption that if we enumerate \( x \) into \( X^{3}_{\gamma_2} \) at stage \( s \) then \( x \in X^{3}_{\gamma_2,s}  \).  While we may simply choose an enumeration of \( X^{3}_{\gamma_2} \) to make this true, we must satisfy \( \mathcal{N}_e \) with respect to the canonical enumeration of elements into c.e.\ sets.  However, using the recursion theorem, we may safely assume that each node is actually in possession of an index for every c.e.\ set built at that node and then, when necessary, we can simply wait until every element enumerated into some \( X^{3}_{\gamma_2} \) appears in it in the canonical enumeration.  Since such modifications are straightforward but tedious, we will refrain from further mention of them.
  
  We assume the nodes that place balls on the machine obey the
  following rules and assumptions.  A node $\alpha \supset \gamma$ can
  only place a ball $x$ on the machine at stage $t$ if
  $x\in X^3_{\gamma_1,t}$.  Moreover, while $\alpha$ might place a ball on
  the machine at stage $s$, $\alpha$ can only place these balls at
  nodes working on the requirement $\mathcal{N}_{e'}$ for some
  $e'$. While we will not restrict how many balls $\alpha$ can place
  on the machine, we assume
  \begin{equation*}
    \tag*{$\mathcal{A}$} \label{eq:A}
    \text{Only finitely many  balls that $\alpha$ places on the machine enter } A.
  \end{equation*}

  Assume that $\gamma \cat 0 \subset f$.  Let $s'$ be such that for
  all $s \geq s'$, $\gamma \subseteq f_{s'}$, $f_{s} \not <_L \gamma$
  and no $\alpha \subset \gamma$ places any more balls into $A$ after
  stage $s'$. Under our extra Assumption~$\mathcal{A}$, we know such a
  stage exists. Assume we are dealing with stages $s \geq s'$.  It is
  not difficult to verify by induction on the length of $\gamma$ that
  if $\gamma \cat 0 \subseteq f_s$, $\curnode{x}{s-1} = \gamma$, and
  $t>s$ is the next stage such that $\gamma \cat 0 \subseteq f_t$ then
  either $x \in A_t$ or $x \in X^3_{\gamma_3,t}$ and for all $y > s'$,
  if $\curnode{y}{t} \subset \gamma$ then $\curnode{y}{t} \cat 1 \subset
  \gamma$.  It is easy to determine which balls enter $A$ between such
  stages.  We assumed that all balls placed on the machine by nodes  $\alpha\subset\gamma$ that enter $A$ have already entered by stage $s'$.  Therefore, the balls that enter $A$ between $s'$ and $t$ come from nodes to the right of $\gamma$.  Since these nodes were reset at stage $s'$, these balls  all have to be larger than $s'$ (otherwise we have that $\curnode{y}{t'} \leq_L \gamma$ for some stage
  $s' < t' < t$) and get into $A$ by stage $t$.  Hence, with the above
  movement of balls and Assumption~$\mathcal{A}$ we have that
  $X^3_\gamma =^* \overline{A}$ and we have met $\mathcal{N}_e$.

  Our next goal is to make $A$ so that it is not computed by any
  $2$-tardy.  We must meet the requirements:
  \begin{equation*}
    \tag*{$\mathcal{P}_e:$}
    \text{If $\Phi_{e_1}(W_{e_2}) = A$, 
      then $W_{e_2}$ is not $2$-tardy.}
  \end{equation*}
  We will assign a parent node $\alpha$ to $\mathcal{P}_e$.  Node $\alpha$ will be
  working on the requirement:
  \begin{equation*}
    \tag*{$\mathcal{P}_\alpha:$}
    \text{If $\Phi_{\alpha}(W_{\alpha}) = A$ 
      then $W_{\alpha}$ is not $2$-tardy.}
  \end{equation*}
Determining whether $\Phi_{\alpha}(W_{\alpha}) = A$ is $\Pi^0_2$. So
  $\alpha$ will have two outcomes $1$ and $2$: outcome $1$ if
  $\Phi_{\alpha}(W_{\alpha}) = A$ and outcome $2$ otherwise.   We will later use
 outcome $0$ to denote a $\Sigma^0_1$ win.  Like above, determining whether
  $\Phi_{\alpha}(W_{\alpha}) = A$ can be measured by asking if there
  are infinitely many expansionary stages where length here measures length of agreement between $\Phi_{\alpha}(W_{\alpha})$ and $A$.

  Assume $\alpha \subseteq f_s$. Let $t \leq s$ be the greatest  stage
  such that $\alpha \subseteq f_t$ (if $t$ does not exist let $t=0$.)
  Define $l_\alpha(s) = \max x [(\forall z <
  x)\Phi^{W_{\alpha,s}}_{\alpha,s}(z)\converge = A_s(z)]$. 
  \noindent
  We say
  that $s+1$ is {\em $\gamma$-expansionary} 
   if
   \begin{enumerate}
   \item      $l_\alpha(s) > l_\alpha(t)$ and, 
   \item for all $\beta \supseteq \alpha$,
  if $x_\beta$ is defined (these will be witnesses to help meet
  requirement $\mathcal{P}$) then $l_\alpha(s) > x_\beta$.
  \end{enumerate}
    If $s+1$ is $\alpha$-expansionary,  
  let $\alpha \cat 1 \subseteq
  f_{s+1}$.  Otherwise, $\alpha \cat 2 \subseteq f_{s+1}$.  If there
  are only finitely many expansionary stages, we need not take any
  action to meet $\mathcal{P}_\alpha$. We only need to take action
  if it appears there are infinitely many expansionary stages (the
  $\Pi^0_2$ outcome).

  We can define the function $p_\alpha(t) = s$ iff $s > t$ is the
  least stage such that $\alpha \cat 1 \subset f_s$. If $\alpha \cat 1
  \subset f$ then $p_\alpha$ is computable. From our work above, we
  know if $\alpha(x,t)\cat 0 \subset \alpha$ then at stage $s = p_\alpha(t)$
  either $x$ is in $A$ or removed from the machine.  This is the
  function we will try to use to witness $W_\alpha$ is not $2$-tardy.

  As a first approximation to showing $W_\alpha$ is not $2$-tardy, we
  might try the following. Above the node $\alpha \cat 1$, we will have
  nodes $\beta$ working on the requirements:
  \begin{equation*}\tag*{$\mathcal{P}_{\alpha,{e}}$:}
    \label{eq:4}
    \begin{split}
      \text{If $X^2_e = \overline{W}_\alpha$ then } &   \\
      \text{there exists $y$ and $s$ such} & \text{ that $ y \in
        X^2_{e,s}$ and } y \in W_{\alpha,p_\alpha(s)}.
    \end{split}
  \end{equation*}
  The idea to meet $\mathcal{P}_{\alpha,{e}}$ is the following:
  At a stage $s$ where $\beta \subset f_s$, choose some large ball
  $x_\beta$. Keep $x_\beta$ out of $A$ and off the machine. 
Let
  $u_{\alpha,s}(x)$ be the use of $\Phi^{W_\alpha}_{\alpha, s}(x)$. 
   Wait for
  a stage $s$ where $\alpha\cat 1 \subseteq f_s$, $X^2_{e,s} \restriction
  u_{\alpha,s}(x_\beta) = \overline{W}_{\alpha,s} \restriction
  u_{\alpha,s}(x_\beta)$ and $\beta \subset f_s$.  If
  such a stage can be found, we want to add $x_\beta$ to $A$ quickly,
  before stage $t = p_\alpha(s)$. Since $\Phi_\alpha^{W_\alpha} = A$, some
  ball $y < u_{\alpha,s}(x_\beta)$ must enter $W_\alpha$ by stage
  $t$. That $y$ must be in $X^2_{e,s}$.

  The problem is adding these balls into $A$ quickly.  If we could
  place the balls $x_\beta$ that we want to enter $A$ into the
  machine at some node $\gamma \subseteq \alpha$ at stage $s$, then by
  our work above we would not have a problem. Since there might be
  infinitely many $\mathcal{P}_{\alpha,{e}}$ that want to
  place balls into $A$, this would violate our extra
  Assumption~\hyperref[eq:A]{$\mathcal{A}$}.  We might try to remove
  this extra assumption. But even so, the set of balls that all
  requirements $\mathcal{P}_{\alpha,{e}}$ might want to add to
  $A$ is not computable. So, we have no reasonable way to manage these
  balls if we allow them all to enter the machine at $\alpha$ or
  below.

  Hence, for each ${e}$ we must assign a different  node $\beta \supseteq \alpha \cat 1$ to
  $\mathcal{P}_{\alpha,{e}}$.  When $\beta$ wants to add
  $x_\beta$ to $A$, the node $\beta$ places $x_\beta$ at the largest
  substring $\gamma=\nu\cat 0$ of $\beta$ where $\nu$ is assigned to some $\mathcal{N}_{e'}$.
  Let stage $t'$ be the first stage that $x_\beta$ goes below $\alpha$
  in the machine. At such a stage we have that $\alpha\cat 1 \subset
  f_{t'}$. If  $X^2_{e,t'} \restriction
  u_{\alpha,t'}(x_\beta) = \overline{W}_{\alpha,t'} \restriction
  u_{\alpha,t'}(x_\beta)$, we let $x_\beta$ continue
  downwards into $A$ for a win (on the above $y$ and $t'$) on
  $\mathcal{P}_{\alpha,{e}}$ as described above.  
 But this may no longer be the
  case.  We have no reason to believe that $t'$ is expansionary for
  $X^2_{e} = \overline{W}_{\alpha}$.  It may be the case that at stage
  $t'$, $X^2_{e,t'}$ is already correctly predicting which balls $y$
  will enter $W_\alpha$.

  Hence, we must modify our requirements to
  \begin{equation*}\tag*{$\mathcal{P}_{\alpha,e}$:}
    \label{eq:5}
    \text{If } X^2_e = \overline{W}_\alpha \text{ and } \neg [(\exists y)
    (\exists s) [ y \in X^2_{e,s} \wedge y \in
    W_{\alpha,p_\alpha(s)}]] 
  \end{equation*}
  \begin{center}
then for all $i$
  \end{center}
  \begin{equation*}\tag*{$\mathcal{P}_{\alpha,e,i}$:}
    \label{eq:6}
    \text{If  } X^2_i = X^2_e \text{ then }
    (\exists y) (\exists s) [ y \in X^2_{i,s} \wedge y \in
    W_{\alpha,p_\beta(s)}]].
  \end{equation*}

  As before some $\beta \supset \alpha \cat 1$ will be assigned to
  $\mathcal{P}_{\alpha,e}$.  The node $\beta$ will have three possible
  outcomes.  The first, $\beta \cat 0$, is in the case we have a
  $\Sigma^0_1$ win for   $\mathcal{P}_{\alpha,e}$, i.e., a ball $y$ and stage $s$ where $y \in
  X^2_{e,s}$ and $y \in W_{\alpha,p_\alpha(s)}$. 
  The second outcome $\beta \cat 1$ holds if
  there is not a $\Sigma^0_1$ win and $X^2_e = \overline{W}_\alpha$.
  The  $\beta \cat 2$ outcome holds otherwise.  As above, we will measure whether $X^2_e
  = \overline{W}_\alpha$ by expansionary stages.

  Assume $\beta \subseteq f_{s+1}$. Let $t \leq s$ be the greatest stage
  such that $\beta\, \cat 1 \subseteq f_t$ (if $t$ does not exist, let $t=0$).
  Define $l_\beta(s) = \max x [(\forall z < x)[X^2_{e,s}(z) =
  \overline{W}_{\alpha,s}(z)]$. 
    We say that $s+1$ is
  {\em $\beta$-expansionary} if 
  \begin{enumerate}
\item   $l_\beta(s) > l_\beta(t)$ and 
\item for
  all $\delta \supseteq \beta$, if $x_\delta$ (a ball to satisfy   $\mathcal{P}_{\alpha,e, i}$) is defined then
  \mbox{$l_\beta(s) >  u_{e,s}(x_\delta)$.}
  \end{enumerate}
  If stage $s+1$ is $\gamma$-expansionary   and we have not seen a $\Sigma^0_1$ win for $\mathcal{P}_{\alpha,e}$, then
  $\beta \cat 1 \subseteq f_{s+1}$. If we have seen the $\Sigma^0_1$
  win, then $\beta \cat 0 \subseteq f_{s+1}$.  Otherwise, $\beta \cat 2
  \subseteq f_{s+1}$.

  If there are only finitely many expansionary stages or we see the
  $\Sigma^0_1$ win, 
  $\mathcal{P}_{\alpha,e}$ is automatically satisfied.  Assume that this is not the case. Hence,
  as above, we are in $\Pi^0_2$ outcome.
 In this case, we must meet $\mathcal{P}_{\alpha,e,i}$, for all $i$.
  For each $i$ we will assign some node $\delta \supseteq \beta \cat
  1$ to $\mathcal{P}_{\alpha,e,i}$. The outcomes and approximations to
  the true path for $\delta$ are defined in similar fashion to what was
  done for $\beta$ and we will not repeat them.  The issue for
  $\delta$ is showing that $\delta$ does not have the $\Pi^0_2$
  outcome, $\delta \cat 1$.

  At a stage $s$ where $\delta \subset f_s$, choose a large unused ball
  $x_\delta$, which we hold out of $A$ and the machine. Wait
  for a stage $s$ where $\delta \cat 1 \subseteq f_s$.  If such a
  stage does not exist we have won this requirement.  If such a stage
  exists, then place $x_\delta$ into the machine at the largest
  substring $\gamma=\nu\cat 0$ of $\beta$ (note, not $\delta$) where $\nu$ is assigned to some
  $\mathcal{N}_{e'}$ and end this stage.

  Now, assuming $\beta \cat 1 \subset f$, there will be a later stage
  $t'$ where $x_\delta$ moves below $\alpha$ and $\alpha \cat 1
  \subseteq f_{t'}$. Otherwise, $\alpha \cat 2 \subset f$ and then the
  action of $\alpha, \beta$ and $\gamma$ are finitary and therefore
  Assumption~\hyperref[eq:A]{$\mathcal{A}$} holds.  If  $X^2_{e,t'} \restriction u_{e,t'}(x_\delta) =
  \overline{W}_{\alpha,t'} \restriction  u_{e,t'}(x_\delta)$,
  we let $x_\delta$ continue downwards into $A$ for a win
 on $\mathcal{P}_{\alpha,{e}}$ as before.  If this
  happens for any $i$, it will provide us with a $\Sigma^0_1$ win on
  $\mathcal{P}_{\alpha,{e}}$ and all the balls $x_\delta$ will be
  removed from the tree since they are to the right of the true path. Therefore, the action of $\beta$ and those
  $\delta \supset \beta \cat 1$ will be finitary. Hence, in this case,
  Assumption~\hyperref[eq:A]{$\mathcal{A}$} holds.

  Assume that $X^2_{e,t'} \restriction  u_{e,t'}(x_\delta) \neq
  \overline{W}_{\alpha,t'} \restriction  u_{e,t'}(x_\delta)$.
  Here, we will remove $x_\delta$ from the machine.  We put $x_\delta$
  into $X^3_{\gamma,3}$ at stage $t'$, for all $\gamma$ such that
  $x_\delta$ is in $X^3_{\gamma,2}$ at stage $t'$.  We now have to
  argue that this provides us with a win for
  $\mathcal{P}_{\alpha,e,i}$.

  \begin{remark}\label{faster}  
    Assume that $\beta\, \cat 1 \subset f$.  Since $p_\alpha(s) \geq s$,
    if it is ever the case that $X^2_{e,s}$ is a proper superset of
    $\overline{W}_{\alpha,s}$ then we know some ball $y$ in $X^2_{e,s}
    \cap {W}_{\alpha,s}$ must later leave $X^2_{e,s}$.  Such a ball
    and a stage will provide us with a $\Sigma^0_1$ win for
    $\beta$. So, we can assume that $\overline{X}^2_e$ is a faster
    enumeration to $W_\alpha$ than the standard enumeration.
  \end{remark}

  We wait for the next stage $t = p_\beta(s)$ such that $\beta \cat 1
  \subseteq f_t$. If such a stage does not exist, then $\beta \cat 2
  \subset f$, the action of $\beta$ and all the related $\delta$ are
  finitary, and therefore, Assumption~\hyperref[eq:A]{$\mathcal{A}$}
  holds.

  At this point we have the following
  \begin{gather*}
    X^2_{i,s} \restriction  u_{e,s}(x_\delta) = X^2_{e,s}
    \restriction  u_{e,s}(x_\delta) = \overline{W}_{\alpha,s}
    \restriction
     u_{e,s}(x_\delta) \\
    X^2_{e,t'} \restriction  u_{e,t'}(x_ \delta) \neq
    \overline{W}_{\alpha,t'} \restriction  u_{e,t'}(x_ \delta)\\
    X^2_{e,t} \restriction  u_{e,t}(x_ \delta) =
    \overline{W}_{\alpha,t} \restriction  u_{e,t}(x_ \delta)
  \end{gather*}

  If $\overline{W}_{\alpha,s} \restriction  u_{e,s}(x_ \delta)
  \neq \overline{W}_{\alpha,t'} \restriction  u_{e,t'}(x_ \delta)$,
  then some $y$ entered $W_{\alpha, t'}$ after stage $s$.  
  Then,  we have a $\Sigma^0_1$ win for
  $\delta$ since $t' < p_\beta(s)$ and $y\in X_{i, s}^2$.  So, assume that $\overline{W}_{\alpha,s} \restriction
   u_{e,s}(x_ \delta) = \overline{W}_{\alpha,t'} \restriction
   u_{e,t'}(x_ \delta)$.  By the fact that $\beta \cat 1 \subset f$ and
  Remark~\ref{faster}, it must be the case that
  \begin{equation*}
    X^2_{e,t'} \restriction  u_{e,t'}(x_ \delta) \subsetneq
    X^2_{e,s} \restriction  u_{e,s}(x_ \delta).
  \end{equation*}
  Hence, some ball $y <  u_{e,s}(x_ \delta)$ must leave
  $X^2_{e,s}$.  Since $X^2_{e}$ is 2-c.e.\ that ball $y$ can never
  return.  Hence, since $\beta \cat 1 \subset f$, that ball $y$ must
  enter $W_\alpha$ and, moreover, it must enter before stage $t=p_\beta(s)$.
  Therefore, we have a $\Sigma^0_1$ win for $ \delta$.

  Assume that $\beta \cat 1 \subset f$.  The infinitely many $\delta$
  above $\beta \cat 1$ might place infinitely many balls onto the
  machine.  Moreover, we can arrange things such that the set of these
  balls is not a c.e.\ set.  But at  most one of these balls will
  enter $A$ and Assumption~\hyperref[eq:A]{$\mathcal{A}$} holds.

  All that remains at this point is to assign the nodes on the tree
  such that all the requirements are met. But this can be done in a
  straightforward fashion.
\end{proof}

\section{Definability and \texorpdfstring{$n$}{n}-tardies}\label{S:Q_n}

We define a property $Q_n$ such that $Q_n$ is nontrivial and if $Q_n(A)$ holds, then $A$ is $n$-tardy.  In the next section, we define a nontrivial  property $\hat{Q}_n$ using $Q_n$ such that if $\hat{Q}_n(A)$ holds, then $A$ is $n$-tardy and $\neg\hat{Q}_i(A)$ holds for all $i<n$.


\begin{definition}
(i) Let $A\subset_\infty C$ denote that $A\subseteq C$ and $C-A$ is infinite.\\
\noindent
(ii) A subset $A$ is a {\em major subset} of $C$, denoted $A\subset_m C$ if $A\subset_\infty C$ and for all $e$,
$$\bar{C}\subseteq W_e \implies \bar{A} \subseteq^* W_e.$$
(If $A\subset_m C$, then $A$ and $C$ are noncomputable.)\\
\noindent
(iii) $A\sqsubset C$ if there exists a $B$ such that $A\sqcup B= C$, i.e., $A\cup B=C$ and $A\cap B=\emptyset$.
\end{definition}

\subsection{\texorpdfstring{$Q_{n}$}{Qn}}  

We begin by defining $Q_{n}$ for \( n \geq 2 \) and proving that if $Q_{n}(A)$ holds, then $A$ is $n$-tardy.
The definition of $Q_{n}$ generalizes Definition~3.2 of $Q_2$ given in \cite{codable-sets-and-orbits-of-computably-enumerable-sets}.
We define $Q_n$ separately for for $n$ even and odd.  

\begin{definition} \hfil\\
	\hfil\\
	\begin{align}\label{E:Q_2n}\tag*{$Q_{2n}(A)$}
	 							&(\exists C\supset_m A) &\\
								&(\forall B_1\subseteq C)(\forall B_2\subseteq B_1)\ldots(\forall B_n\subseteq B_{n-1}) \notag & \\
							  &(\exists D_1\subseteq C)(\exists D_2\subseteq D_1)\ldots(\exists D_n\subseteq D_{n-1}) \notag &\\
								&(\forall S\sqsubset C)(\exists T_1\supseteq \bar{C})(\exists T_2\subseteq T_1) \ldots (\exists T_{n}\subseteq T_{n-1}) \notag &\\ 
	 & \quad \label{E:Q_2nif}\tag*{$Q_{2n}(A):\text{if}$}\left[\begin{aligned}
		B_1\cap(S-A) &= D_1\cap(S-A) \\
		B_2\cap(S-A) &= D_2\cap(S-A) \\
		 &\ldots  \\
		B_n\cap(S-A) &= D_n\cap(S-A) 
	\end{aligned} \right]  \\
	 			  & \notag \qquad \qquad \qquad \qquad \implies \\
		  	&	\quad \label{E:Q_2nthen}\tag*{$Q_{2n}(A):\text{then}$}  \left[\begin{alignedat}{3}
			 (A\cup T_2)&\cap(S\cap T_1) &&= & B_1 &\cap(S\cap T_1) \\
			 (A\cup T_3)&\cap(S\cap T_2) &&= & B_2 & \cap(S\cap T_2)\\
			 & &&\ldots \\
			 (A\cup T_n)&\cap(S\cap T_{n-1}) &&= & B_{n-1} &\cap(S\cap T_{n-1}) \\
			 A&\cap(S\cap T_n) &&= & B_n &\cap(S\cap T_n) 
			 \end{alignedat}\right] 
	\end{align}

\medskip

\begin{equation*}\label{E:Q_2n+1}\tag*{$Q_{2n+1}(A)$}
	(\exists Y \subseteq \bar{A})Q_{2n}(A \cup Y)
\end{equation*}

\end{definition}

\begin{theorem}\label{T:Q_2nimptardy} If $Q_{n}(A)$ holds, then $A$ is $n$-tardy.

\end{theorem}

We break this proof into two lemmas, one handling the case where \( n \) is even and the other handling the case where \( n \) is odd.

\begin{lemma}\label{lem:Q_2n1imptardy}
	If \( Q_{2n}(A) \) implies \( A \) is \( 2n \)-tardy for any c.e.\ set $A$, then \( Q_{2n+1}(A) \) implies \( A \) is \( 2n+1 \)-tardy for any c.e.\ set $A$.
\end{lemma}
\begin{proof}
	If \( Q_{2n+1}(A) \) then  \( Q_{2n}(A \cup Y) \) holds for some \( Y \) disjoint from \( A \).  By assumption \( A \cup Y \) is \( 2n \)-tardy.  Thus, if \( p(s) \) is a total computable function, there is some \( 2n \)-c.e.\ set \( X^{2n}=X_1 - X_2 + X_3 - \ldots - X_{2n} \) equal to \( \overline{A \cup Y} \) such that \( x \in X^{2n}_s \implies x \not \in A_{p(s)} \).  Let 
	\begin{center}
	\( X^{2n+1}=X^{2n}+Y=X_1 - X_2 + X_3 - \ldots - X_{2n} + Y \).  
	\end{center}
	Since \( Y \cap A = \emptyset \),  \( X^{2n+1} = \bar{A} \) and \( x \in X^{2n+1}_s \implies x \not \in A_{p(s)} \).  Since \( p(s) \) was arbitrary, \( A \) is \( 2n+1 \)-tardy.
\end{proof}

\begin{lemma}\label{lem:Q_2nimptardy}
	$Q_{2n}(A)$ implies $A$ is $2n$-tardy.
\end{lemma}

This is a generalization of Theorem~3.3 in \cite{codable-sets-and-orbits-of-computably-enumerable-sets} and we retain their approach but  present it as a modern \( \Pi^0_2 \) guessing argument.

\begin{proof}

Fix  $A$ and $C$ (and indexes for them) such that $A$ satisfies ${Q}_{2n}(A)$ via $C$ and $A \subseteq C \searrow A$ where the latter property can be guaranteed purely by change of index.    Following  \cite{codable-sets-and-orbits-of-computably-enumerable-sets}, we think of \( {Q}_{2n}(A) \) as a two player game between the $\exists$-player (called EXISTS) who plays the sets \( \vec{D}= (D_1, D_2, \ldots D_n) \) and \( \vec{T}=(T_1, \ldots, T_n) \) and the $\forall$-player (called FORALL) who plays the sets $B_1, B_2, \ldots B_n$ and $S \sqsubset C$.  Should \( A, C, S, \vec{D}, \vec{T} \) witness the satisfaction of \( {Q}_{2n}(A) \) we say the EXISTS player wins.  Otherwise, the FORALL player wins.  Since \( C \) witnesses the satisfaction of \( {Q}_{2n}(A) \), the EXISTS player must have a winning strategy.  Given any total computable function \( p(s) \), the proof will proceed by specifying a strategy for the FORALL player such that winning response \( \vec{D}, \vec{T} \) of the EXISTS player allows us to build a \( 2n \)-c.e.\ set \( X^{2n} \) witnessing that \( A \) is \( 2n \)-tardy.


Given a total computable function \( p(s) \),  FORALL will respond by building \( \vec{B} \).  However, in the construction of \( B \), FORALL will want to use information about the particular sets \( \vec{D}, \vec{T} \) played by EXISTS,  but \( \vec{B} \) must be built without knowledge of \( \vec{D} \) or \(  \vec{T} \).  We let \( \vec{B} \) react to the particular choice of \( \vec{D} \) by simultaneously building \( \vec{B} \) and a sequence of sets \( S_{e} \sqsubset C \) such that on \( S_{e} \), the collection \( \vec{B} \) plays against \( \vec{D_e}=(W_{e_0},W_{e_1},\ldots,W_{e_n})\).  During this construction, \( \vec{B} \) will be built so that, for every \( e \), property \eqref{E:Q_2nif} holds for \( S=S_{e}, \vec{D}=\vec{D_e} \). Thus, for EXISTS to have a winning strategy, there must be some \( \vec{T} \) witnessing the satisfaction of \eqref{E:Q_2nthen}.

We now further divide up the sets \( S_{e} \) into the sets \( S_{e,j} \) with \( S_{e}=\bigsqcup_{j \in \omega} S_{e,j} \) so that FORALL  builds \( \vec{B} \) to play against \( \vec{D_e}, \vec{T_j} \) on \( S_{e,j} \).  Since \( S \) must be played without knowledge of \( \vec{T} \), we appear to run the risk that the winning strategy for EXISTS never plays \( \vec{T_j} \) against \( S_{e,j} \).   However, since \( \vec{B}, \vec{D_e}, S_{e} \) satisfies \eqref{E:Q_2nif}, there is some \( j \) such that  \(\vec{B}, \vec{T_j},  S_{e} \) satisfy \eqref{E:Q_2nthen}.  But as \( S_{e,j} \subset S_e \), it follows that \(\vec{B}, \vec{T_j},  S_{e,j} \) satisfy \eqref{E:Q_2nthen}.  Thus, provided for all \( e \) we maintain \eqref{E:Q_2nif} for \( S_{e}, \vec{D_e} \),  we may assume that for some \( e,j \) the sets \(\vec{B}, \vec{D_e}, \vec{T_j},  S_{e,j} \) satisfy both \eqref{E:Q_2nif} and \eqref{E:Q_2nthen}.

We let \( \alpha \) range over indexes \( e,j \) for \( n \) tuples of c.e.\ sets and define

\begin{align*}
	S_\alpha &= S_{e,j}\\
	D^{\alpha}_i&=W_{e_i}\\
	T^{\alpha}_i&=W_{j_i}\\
	\shortintertext{where we stipulate that our indexes satisfy}
	D^{\alpha}_1 &\subseteq C\\
	D^{\alpha}_{i+1} &\subseteq D^{\alpha}_{i}  \\
	T^{\alpha}_{i+1} &\subseteq T^{\alpha}_{i}  
\end{align*}

Relative to a particular choice of \( \vec{B} \), the predicate \( F(\alpha) \) asserting that the sets \(\vec{B}, \vec{D_e}, \vec{T_j},  S_{e,j} \) satisfy both \eqref{E:Q_2nif} and \eqref{E:Q_2nthen} is \( \Pi^0_2 \).  Thus, there is a uniformly computable sequence of predicates \( F_s(\alpha) \) referring only to the commitments we have made about \( \vec{B} \) by stage \( s \) in our construction such that \( F(\alpha) \leftrightarrow (\exists^{\infty}s)\ F_s(\alpha) \).  Using this predicate, we define a strong array of finite sets \( U^{\alpha}_{k} \) for every \( \alpha \) and \( k \in [1,n] \) as follows. 

\begin{align*}
x\in U^{\alpha}_{1, s} &\iff x\in U^{\alpha}_{1, s-1} \ \lor \ [\ s \geq x \land x\in (T^{\alpha}_{1, s}- C_s)  \land F_s(\alpha) \ ].\\
x\in U^{\alpha}_{l, s} &\iff x\in U^{\alpha}_{l-1, s}\cap T^{\alpha}_{l, s}\\
\intertext{By way of the slowdown lemma \cite{SoareBook} applied to the above arrays, we define}
X^{\alpha}_{2i-1} &= \bigcup_{s \in \omega} U^{\alpha}_{i,s}\\
\shortintertext{satisfying} X^{\alpha}_{2i-1,s} &\subset U^{\alpha}_{i,s}
\end{align*}

If we build \( S_{e,j} \) as described, there must be some least \( \alpha \) for which \( F(\alpha) \) holds by the remarks above. For that \( \alpha \), we have \( X^{\alpha}_1 \supset^{*} \bar{A} \) since \eqref{E:Q_2n} requires that \( T_1 \supset \bar{C} \) and \( F_s(\alpha) \) holds for infinitely many \( s \). So, \( U^{\alpha}_{1} \supset \bar{C} \) and \( A \subset_m C \). We also have \( T_{i+1} \subseteq T_i \) and \( X^{\alpha}_{2i-1} = T^{\alpha}_i \cap X^{\alpha}_{1} \) by definition. So, if the sequence \( \vec{T^{\alpha}} \) witnesses that \eqref{E:Q_2nthen} holds, we may replace each \( T^{\alpha}_i \) with \( X^{\alpha}_{2i-1} \) without falsifying \eqref{E:Q_2nthen}.

We now build \( S_{\alpha} \) with the intention that (with finitely many exceptions) every element that is in \( X^{\alpha}_1 \cap A \) is in \( S_\alpha \).   If \( x\in C_{s+1}-C_s \), take the least \( \alpha \) such that \( x\in U^{\alpha}_{1, s} \) and enumerate \( x \) into \( S_{\alpha} \).  If no such \( \alpha \)  exists, enumerate \( x \) into the garbage set \( S_{-1} \).  Note that  \( C= \bigsqcup_{\alpha\in 2^{<\omega}} S_{\alpha} \sqcup S_{-1} \) by construction, so,  \( S_{\alpha} \sqsubset C \) for every \( \alpha \).  Furthermore, by construction, once \( x \) enters \( C \) it can no longer enter \( U^{\alpha}_{1} \) for any \( \alpha \). Suppose \( \alpha \) is the least such that \( F(\alpha) \) holds. Since \( U_1^{\beta} \) is finite  for every \( \beta < \alpha \), we have  \( U^{\alpha}_{1} \cap C \subset^{*} S_{\alpha} \).  Hence,  for all \( i \in [1,n] \)

\begin{align}\label{E:XoddsubSalpha}
X^{\alpha}_{2i-1} \cap C \subset^{*} S_{\alpha}. 
\end{align}
 Conversely, \( S_\alpha \subseteq X^{\alpha}_1 \).  We are now ready to define  \( \vec{B} \) and the even components of \( X^{\alpha} \).
Let 
\begin{align*}
	X^{\alpha}_{2i}&= S_{\alpha} \cap D^{\alpha}_{i}\\
	\shortintertext{where by way of the slowdown lemma \cite{SoareBook}, we ensure that}
	X^{\alpha}_{2i+2} &\subseteq X^{\alpha}_{2i}\searrow X^{\alpha}_{2i+2}
\end{align*}

Since \( X^{\alpha}_{1}\ \cap\ C\subset^{*} S_{\alpha} \) and \( X^{\alpha}_{1}\ \cap\ \overline{C}\subseteq \bar{A} \),  requiring \( X^{\alpha}_{2i} \) to be a subset of \( S_{\alpha} \) is no handicap to ensuring $X^\alpha=\overline{A}$.  If \( F(\alpha) \) holds, then we claim that 

\begin{equation}\label{E:even-goes-in}
	\begin{aligned}
		& X^{\alpha}_{2j-1} \cap X^{\alpha}_{2j} \cap \bar{A} \subseteq X^{\alpha}_{2j+1}\\ 
		& X^{\alpha}_{2n-1} \cap X^{\alpha}_{2n} \cap \bar{A}=\emptyset.
	\end{aligned}	
\end{equation}

To see this, let \( x \in X^{\alpha}_{2j-1} \cap X^{\alpha}_{2j} \cap \bar{A}  \).  As \( X^{\alpha}_{2j} = D^{\alpha}_{j} \cap S_\alpha, \) we have  \( x \in D^{\alpha}_{j} \cap (S_\alpha - A) \) which by \eqref{E:Q_2nif} is contained in \(  B_j \).  By a prior remark, we may substitute \( X^{\alpha}_{2j-1} \) in for \( T_j \) in \eqref{E:Q_2nthen} and as \mbox{\( x \in B_j \cap S_\alpha \cap X^{\alpha}_{2j-1} \),} we have \mbox{\(x \in  A \cup X^{\alpha}_{2j+1} \).}  Since  \( x \not\in A \), we have \( x \in X^{\alpha}_{2j+1} \).  Moreover, by similar reasoning, $X^{\alpha}_{2n-1} \cap X^{\alpha}_{2n} \cap \bar{A}=\emptyset$. We then derive the following containments.

\begin{equation}\label{E:barAin}
\begin{aligned}
	\bar{A} \cap S_{\alpha} &\subseteq X^{\alpha} \\ 
	\bar{A} &\subseteq^{*} X^{\alpha} 
\end{aligned}
\end{equation}

For the first containment, if \( x \in \bar{A} \cap S_{\alpha}  \) then, as \( S_{\alpha} \subseteq X^{\alpha}_1 \), there is a maximal \( j \) such that \( x \in X^{\alpha}_{2j-1} \).  Since the even indexed components of \( X^{\alpha} \) are nested, if \( x \not\in X^{\alpha}_{2j} \) then \( x \in X^{\alpha} \), and we are done. If \( x \in X^{\alpha}_{2j} \), then \eqref{E:even-goes-in}  yields a contradiction. The second containment follows since \( X^{\alpha} \supseteq X^{\alpha}_1 - S_{\alpha} \) (by definition, each $X^\alpha_{2j}\subseteq S_\alpha$, so no elements outside of $S_\alpha$ are removed from $X^\alpha$)  and \( X^{\alpha}_1 \supseteq^{*} \bar{A} \).  We now define \( \vec{B} \) to that the other direction of containment and the tardiness property hold.

\begin{equation}\label{E:Bdef}
	x \in B_i \iff (\exists \alpha)(\exists s)\left[ x \in X^{\alpha}_{2i,s} \land x \not\in A_{p(s)}  \right]
\end{equation}

Tracing out the definition of \( X^{\alpha}_{2i} \), it is evident that on \( S_{\alpha} - A \) we have \mbox{\( B_i = D^{\alpha}_i=W_{e_i} \).}   Hence, by our earlier arguments,  there is  some \( \alpha \) such that \( F(\alpha) \) holds.  Now let \( \alpha \) be the least such.  Since \( B_i \cap S_{\alpha} \subseteq D^{\alpha}_i \cap S_{\alpha} \), using \eqref{E:Q_2nthen} we see

\begin{equation*}
	 A \cap X^{\alpha}_{2i-1} \cap S_{\alpha} \subseteq B_i \cap X^{\alpha}_{2i-1} \cap S_{\alpha} \subseteq D^{\alpha}_i \cap S_\alpha = X^{\alpha}_{2i}.
\end{equation*}

Thus, if \( x \in A \cap S_{\alpha} \) then \( x \not\in X^{\alpha} \).  By \eqref{E:XoddsubSalpha}, \( X^{\alpha} \cap C \subseteq^{*} S_{\alpha} \).  Since \( X^{\alpha} \cap C \subseteq^{*} S_{\alpha} \) and \( A \subseteq C \), this entails \( \bar{A} \supseteq^{*} X^{\alpha}  \).  Putting this together with \eqref{E:barAin}, we conclude

\begin{align*}
	\bar{A} \cap S_{\alpha} &= X^{\alpha} \cap S_{\alpha} \\
	\bar{A} &=^{*} X^{\alpha} 
\end{align*}

We now argue that \( X^{\alpha} \) has the desired tardiness properties.  Suppose \( x \in X^{\alpha}_1 \) and \( x \in A \cap S_{\alpha} \).  Let \( j \) be the greatest such that \( x \in X^{\alpha}_{2j-1}  \).  Now suppose \( x  \) enters \( X^{\alpha}_{2j} \) at stage \( s \).  If \( x \in A_{p(s)} \) then by \eqref{E:Bdef} \( x \not\in B_i \). But as \( x \in X^{\alpha}_{2j-1} \cap S_{\alpha} \), it follows from \eqref{E:Q_2nthen} that \( x \not\in A \).  This is a contradiction.  Therefore, 

\begin{equation*}
	x \in S_{\alpha} \cap X^{\alpha}_s \implies x \not\in A_{p(s)}
\end{equation*}

Now set \( X_{2i}= X^{\alpha}_{2i} \) and \( X_{2i-1}=^{*} X^{\alpha}_{2i-1} \) where we build \( X_{2i-1} \) by removing the finitely many members by \eqref{E:XoddsubSalpha} of \( A \cap \bar{S_\alpha} \cap X^{\alpha}_1 \)   from \( X^{\alpha}_{2i-1}  \) and adding the finitely many members of \( \bar{A} - X^{\alpha}_1 \).  The set \( X \) witnesses that \( A \) is \( 2n \)-tardy with respect to \( p(s) \).   Since \( p(s) \) was arbitrary, we can conclude \( A \) is \( 2n \)-tardy.

\end{proof}

Taken together these lemmas suffice to establish Theorem \ref{T:Q_2nimptardy}.

\section{Proper Satisfaction \texorpdfstring{\( Q_n \)}{Q_n}}\label{S:Q_nsat}

At this point we have a countable collection of properties \mbox{\( Q_n\) for \(n \geq 2 \)} preserved under automorphism guaranteeing incompleteness.  It is easily verified that \( Q_{n}(A)\)  implies \(Q_{n+1}(A) \) so to illustrate countably many incomplete orbits, we must show this hierarchy of properties does not collapse.  In particular, it will suffice to show that for every \( n > 2 \) there is a properly \( n \)-tardy \( A \) satisfying \( Q_n(A) \) as we can then define 
\begin{equation*}
	\hat{Q}_n(A) \iff Q_n(A) \land \lnot Q_{n-1}(A) \land \ldots \land \lnot Q_2(A)
\end{equation*}

Since by Theorem \ref{T:Q_2nimptardy}, any set satisfying \( Q_n(A) \) must be \( n \)-tardy it follows that our set \( A \) satisfies \( \hat{Q}_n(A) \) and that the properties \( \hat{Q}_n\) for \( n \geq 2 \) give countably many disjoint orbits.

\begin{theorem}\label{T:Q_nsatisfied} 
	For all \( m \geq 2 \) there is a properly \( m \)-tardy \( A \) satisfying \( Q_m(A) \).
\end{theorem}

Again we consider the even and odd cases separately.  We first work to show that there is a properly \( 2n \)-tardy satisfying \( Q_{2n} \) and then modify this argument to yield a properly \( 2n+1 \)-tardy satisfying \( Q_{2n} \).     

Ideally our argument in the even case would establish that any \( 2n \)-tardy that has a major superset satisfies \( Q_{2n} \), but this appears to be insufficient.  The role played by \( C \) is not only that of a major superset but also provides an early warning that something in \( T_1 \) may threaten to enter \( A \) (i.e., enter \( T_2 \)).  That is, we must wait until elements enter \( C \) before we can target them for entry into \( A \).  Thus, we first construct sets \( A \) and \( C \) with these properties.

\subsection{Building \texorpdfstring{\( A \)}{A} and \texorpdfstring{\( C \)}{C}}

\begin{lemma}\label{lem:major-2n-tardy}
	For every \( n \geq 1 \) there is a properly \( 2n \)-tardy set \( A \) and a c.e.\ set \( C \) with \( C\supset_m A \) such that for every total computable function \( p \) there is a \( 2n \)-c.e.\ set \( X_e^{2n} \) such that
	\begin{subequations}\label{E:A-and-C}
	\begin{align}
		& \overline{A} = X_e^{2n}  = (X^{2n}_{e_{1}} - X^{2n}_{e_{2}}) \cup \ldots \cup (X^{2n}_{e_{{2n-1}}} - X^{2n}_{e_{{2n}})} \\
		& (\forall x,s) \left( x \in X^{2n}_{e, s} \implies x \not\in A_{p(s)}  \right) \\
		& (\forall k < 2n) \left[ X^{2n}_{e_{k+1}} = C\searrow X^{2n}_{e_{k+1}} \right]\label{E:A-and-C:C-first}\\
		& X^{2n}_{e_1} \supseteq  X^{2n}_{e_2} \supseteq \ldots \supseteq X^{2n}_{e_{2n-1}} \supseteq X^{2n}_{e_{2n}} \label{E:A-and-C:nested} \\
		& (\forall i < 2n)[ X^{2n}_{e_{i+1}} = X^{2n}_{e_{i}} \searrow X^{2n}_{e_{i+1}} ] \label{E:A-and-C:nicely-nested}
	\end{align}
\end{subequations}
\end{lemma}

\begin{proof}[Proof of Lemma \ref{lem:major-2n-tardy}]

To prove the claim, we start with a simple set \( \hat{C} \) with the property that \( \lvert\hat{C}\restriction_{2x}\rvert \leq x \) and simultaneously construct \( C \supseteq \hat{C} \) and \( A \).  During this construction, we  refer to the index of \( C \) as a c.e.\ set so we can measure its speed of enumeration.  We justify this circularity by regarding the construction as a computable function operating on a guess at the index for \( C \) and returning an index for the resulting  \( C \) we build and then applying the recursion theorem.  To effect the construction of \( A \),  we will work to meet the requirements \( \mathcal{N}_e, \mathcal{M}_e, \mathcal{R}_e \) specified below to which we assign priorities \mbox{\( 3e, 3e+1, 3e+2 \)}, respectively.  These requirements are thought of as being laid out vertically in order of priority.  Ultimately, the true construction will take the form of a tree argument in the style of \ref{T:3tardy}. Rather than repeat the standard details of the tree layout, we instead present the argument as if it were a infinite injury pinball argument so as not to hide the real work in the magic of the tree machinery.  Ultimately, however, we will observe that the computable corrections required by infinite injury can simply be considered as the action of the tree when phrased as a \( \Pi^0_2 \) tree argument and can thus be squared with requirements \eqref{E:A-and-C:nested} and \eqref{E:A-and-C:nicely-nested}.  Thus the construction may be regarded as a pinball machine with \( A \) at the bottom of the machine and the requirements stretching upwards.

During the construction, balls (numbers) will be released at requirements of the form \( \mathcal{M}_e \) and \(  \mathcal{R}_e \).  These balls attempt to flow down through the negative requirements below.   When (and if) they reach the bottom, they are enumerated into \( A \).   The organization of this construction is in principle similar to that used before but the extra complexity of a full \( \Pi^0_2 \) tree construction is unnecessary so we abandon it for clarity.

Each negative requirement \( \mathcal{N}_e \) will construct a \( 2n \)-c.e.\ set \( X^{2n}_{e} \) in the hope of meeting the requirement below, falling short only by virtue of computable injury.  By pausing the construction until elements appear in the canonical enumeration of \( C \), we may insist that \eqref{E:A-and-C:C-first} holds.

\begin{equation*}\tag*{\(\mathcal{N}_e\)}
  \label{eq:neg-Q_n-satisfied}
	\begin{aligned}
    (\exists x)&(\varphi_e(x)\mathpunct{\uparrow})  \\
		& \mbox{or} \\
  [X^{2n}_e  = \overline{A}] &\ \&\  (\forall x)(\forall s)[x \in X^{2n}_{e,s}  \implies x \not\in A_{\varphi_e(s)}]
\end{aligned}
\end{equation*}

We act to meet this requirement as follows.  At the start of stage \( s > e \), fix \( l \) to be maximal such that \( (\forall x < l)[\varphi_{e, s}(x)\mathpunct{\downarrow}] \) and put every \( x < l \) into \( X_{e_1} \) that is not already in \( A \) or located below \( \mathcal{N}_e \) along our list of requirements.  If a ball \( x \) targeted for \( A \) by a weaker priority requirement reaches \( \mathcal{N}_e \) at stage \( s \) and it is not yet in \( X_{e_1} \), it is immediately allowed to fall through to the next negative requirement along the path to \( A \).  Otherwise, if \( x \in X_{e_1} \) let \( j \) be the largest index such that \( x \in X_{e_j} \).  Place \( x \) into \( X_{e_{j+1}} \) and delay \( x \) from falling through to the next negative requirement until the first stage \( t \) such that  \( \varphi_{e, t}(s)\mathpunct{\downarrow}\) is reached.  Whenever a weaker requirement decides to cancel its attempt to place some \( x \in X_{e_{2j}} -  X_{e_{2j+1}} \) into \( A \) before completion, \( x \) is placed into \( X_{e_{2j+1}} \).     Note that if \( \varphi_e \) is partial \( X_{e_1} \) will be finite and only finitely many balls will be permanently delayed by \( \mathcal{N}_e \). 

This construction suffices to meet \( \mathcal{N}_e \) modulo the balls put into \( A \) by higher priority requirements.  At the end of the construction, we will observe that the set of elements that so slip past \( \mathcal{N}_e \) is computable. Thus, we can modify \( X^{2n}_e \) to satisfy the requirement without sacrificing any of the desired properties.  We now work to ensure that \( C \supseteq_m A \) via the following requirement.

\begin{equation*}\tag*{\(\mathcal{M}_e\)}
  \label{eq:major-Q_n-satisfied}
    W_e \supseteq \overline{C} \implies W_e \supseteq^{*} \overline{A}
\end{equation*}

Given \( C \) that we are building,  \( A \subset_m C \) requires that if \( W_e \supset \overline{C} \) then \mbox{\( C - A \subseteq^{*} W_e \).}  If we knew from the outset that \( W_e \cap C \) was infinite and \( W_e \supseteq \overline{C} \), we would ensure  all but finitely many members of \( C -A \) are in \( W_e \) by enumerating elements into $A$.  Fixing \( C \) to be a simple set ensures that if \( W_e \supseteq \overline{C} \) then \( W_e \cap C \) is infinite.
  Thus, taking \( C\) simple would suffice to satisfy \( \mathcal{M}_e \).  Since we cannot determine whether \( W_e \) is infinite, we instead  assume that  we have seen the entirety of \( W_e \) and correct our construction if we see more elements enter \( W_e \). 
   In particular, if we see \( W_e \) extend to contain \( \overline{C} \cap [0,l] \), we can respond by enumerating (almost) every \( x <l \) with \( x \in C \) into \( A \) to keep \( C - A \subseteq^{*} W_e \).  To ensure \( C - A  \)  is infinite, 
we absolve the first \( e \) (candidate) members of \( C - A \) from being affected by \( \mathcal{M}_e \).

We fix $\{C_s\}_{s\in\omega}$, a stage wise approximation to \( C \),  such that \( C_0 \) is some infinite computable subset of \( \hat{C} \) and other elements enter \( C_s \) only when they are enumerated into \( \hat{C} \) or placed into \( C \) by \( \mathcal{R}_e \), which we describe below.  We also fix a countable collection of markers \( m_k \) shared across all the \( \mathcal{M}_e \) requirements whose position at stage \( s \) we denote by \( m_{k,s} \) with the intention (which we almost fulfill) of letting them come to rest on \( C - A \).  We describe the motion of these markers in terms of an \( e \)-state construction.  Instead of maximizing the \( e \)-state of our markers, which would guarantee that any c.e.\ set containing infinitely many elements from \( C - A \) contains almost all of them, we only maximize the \( e \)-state for c.e.\ sets threatening to contain \( \overline{C} \).  To this end, we employ the following twist on the notion of an \( e \)-state.

\begin{definition}\label{def:Cestate}
	The \( C \)-complementing \( e \)-state of \( x \), denoted \( \epsilon^C_e(x) \), is defined to be the \( \nu \in 2^{e-1}  \) such that 
	\begin{equation*}
		(\forall i < e) \left(\nu(i)= 1 \leftrightarrow W_{i} \supseteq (\overline{C} \cap [0,x]) \land x \in W_{i}\right)
	\end{equation*}
	The \( C \)-complementing \( e \)-state of \( x \) at stage \( s \), \( \epsilon^C_e(x,s) \) is defined to be the \( \nu \in 2^{e-1}  \) such that
	\begin{equation*}
			(\forall i < e) \left(\nu(i)= 1 \leftrightarrow W_{i,s} \supseteq (\overline{C_s} \cap [0,x]) \land x \in W_{i,s}\right)
	\end{equation*}
\end{definition}

At the start of the construction, we place \( m_k \) on the \( k \)-th element of \( C_0 \).  At the start of stage \( s +1 \) for every \( x \in C_{s+1} - C_{s} \), we pick the least \( k \) such that \( m_{k,s} > x \) and shift the markers after \( m_k \) down to their predecessors' location to fill the gap.  Note that since \( m_{k,s} > x \) if \( \epsilon^C_e(m_{k,s},s)(i) = 1 \) then since \( x \not\in C_{s} \) it follows that \( x \in W_{i,s}  \) and thus \( \epsilon^C_e(m_{k,s+1},s+1) \geq_L  \epsilon^C_e(m_{k,s},s)\) where \( \geq_L \) denotes the lexicographic order. 

After all the requirements with greater priority than \( \mathcal{M}_e \) have acted at stage \( s \), we search for the least \( k, k' \) with \( k' > k \geq e \) and 
\begin{equation} \epsilon^C_{e+1}(m_{k',s},s) >_L \epsilon^C_{e+1}(m_{k,s},s) .  
\end{equation}  We then move the marker \( m_k \) to the location occupied by \( m_{k'} \).  We  shift the later markers up accordingly and target the locations previously occupied by \( m_j \) for \( k \leq j < k' \) for entry into \( A \).    Note that some of these elements may be reserved by higher priority requirements described later.  In this  case, we still move the markers but respect the reservation.  Likewise, the motion of the markers is not affected if one of the elements targeted for \( A \) is permanently restrained by a negative requirement.

We inductively argue that each marker comes to rest.  Pick \( s \) large enough so that every \( m_{k'}\) for  \(k' < k \) has already come to rest on its final position and then select \( s' > s \) to maximize  \(\epsilon^C_{k+1}(m_{k,s'},s')  \) for \( s' \geq s \).  The marker \( m_k \) cannot be moved at this point unless new elements are enumerated into \( C \).  By the above remarks, this movement cannot decrease  \(\epsilon^C_{k+1}(m_{k,s'},s')  \). Eventually, no further elements of \( C \) are enumerated below \( m_k, \) and the marker comes to rest.  We let  \( m_{k,\infty} \) denote the location that  \( m_k \) settles upon permanently. 

We now argue that \( \mathcal{M}_e \) is satisfied if all but finitely many of the elements targeted  for \( A \) by \( \mathcal{M}_e \) eventually enter \( A \).  To see this, fix some  \( W_e \supseteq \overline{C} \) and note that the intersection of all sets \( W_i \) for \( i \leq e \) such that \( W_i \supseteq \overline{C} \) contains \( \overline{C} \).  By the simplicity of \( C \), this intersection must have an infinite intersection with \( C \). It follows that all but finitely many elements of \( C- A \) are contained in \( W_e \) and, indeed, all but finitely many elements in \( C - A \) have some \( C \)-complementing \( (e+1) \)-state \( \nu \).  Moreover, those elements targeted for \( A \) by \( \mathcal{M}_e \) form a computable set as, for large enough \( x \), \( \mathcal{M}_e \) targets \( x \) for \( A \) only if it has done so by the time we see a marker above \( x \) attain the \( C \)-complementing \( e \)-state \( \nu \).

Lastly, we must also guarantee that \( A \) is not \( (2n-1) \)-tardy by meeting the requirements \( \mathcal{R}_e \).  To that end, we ensure every \( (2n-1) \)-c.e.\ set  fails to provide a tardy approximation to \( \overline{A} \).  We fix a monotonically increasing computable function \( p(s) \) such that if our construction directs us to place \( x \) into \( A \) at stage \( s \) then \( 
B_{p(s)} \)  
for $B$ a c.e.\ set given in the the canonical enumeration.  Provided \( B=A \) then \( p(s) \) will be total.    By meeting the following requirements, we guarantee that no \( (2n-1) \)-c.e.\ approximation \( Y^{2n-1}_e \) to \( \overline{B} \) can work \( p(s) \) steps ahead of \( B \) if \( B=A \).  Since the requirements dealing with \( B=\REset{i} \) do not interact significantly with those requirements working against \( B=\REset{i'} \), we drop the subscript \( i \) from the statement of the requirement.

\begin{equation*}\tag*{$\mathcal{R}_e$}\label{eq:properly-Q_n-satisfied}
     (\exists x)\left[ 
			\begin{aligned}
		 		(\exists s)[ x  \in Y^{2n-1}_{e,s}\  \land\  &  x \in (A_{s} - A_{s-1})]  \\
		 		& \lor  \\
		 	Y^{2n-1}_e(x)  \neq\overline{A}(x) &
			\lor A \neq B
		\end{aligned} \right]
\end{equation*}

So long as we believe \( \mathcal{R}_e \) to be unsatisfied, our strategy will work to provide \( \mathcal{R}_e \) another ball \( x \) it alone controls and let \( \mathcal{R}_e \)  hold \( x \) out of \( A \) until we later see \( Y^{2n-1}_{e} \) change its mind to guess that  \( x \in \overline{B} \).  At this point, we target $x$ for entry into \( A \), and if $x$ gets into \( A \) before it leaves \( Y^{2n-1}_{e} \), we satisfy \( \mathcal{R}_e \).  If instead \( x \) leaves \( Y^{2n-1}_{e} \) before \( x \) enters \( A \), we simply cancel our targeting of \( x \) for \( A \), returning it back to wait at \( \mathcal{R}_e \) and again hold it out of \( A \) until the situation changes again.  Since \( Y^{2n-1}_{e} \) can change its mind one less time than $A$ can,  \( Y^{2n-1}_{e} \) eventually  must either fail to respond to these threats or \( x \) must enter the \( (2n-1) \)-th component of \( Y^{2n-1}_{e} \).  In the second case,  we may safely put \( x \) into \( A \) since if \( Y^{2n-1}_{e} \) really witnessed the tardiness of \( B \), we must have \( x \nin B \).  Note that we only place a new element into $A$ after placing it in $C$ if it is not already present there.  The real complexity in meeting this requirement is guaranteeing that we will eventually reserve some \( x \) large enough that  $x$  is not permanently restrained by higher priority negative requirements and $x$  has a large enough \( C \)-complementing \( e \)-state so that it is not  co-opted by any higher priority \( \mathcal{M}_e \).

During the construction, each \( \mathcal{R}_e \) will reserve a finite collection of intervals \( \{ \left[ l^e_i, h^e_i \right] \}_{i=1}^{n_e} \)  for its exclusive use such that if \( e \neq e' \) or \( i \neq i' \) then  \( \left[ l^e_i, h^e_i \right] \) and \( \left[ l^{e'}_{i'}, h^{e'}_{i'} \right] \) are disjoint.  Inside each interval, \( \mathcal{R}_e \) will maintain a marker \( r^e_i \) with position \( r^e_{i,s} \)  at stage \( s \) to indicate the currently active element in that interval.  Only \( \mathcal{R}_e \) or a higher priority \( \mathcal{M}_{e'} \) is allowed to target a member of \( \left[ l^e_i, h^e_i \right] \) for \( A \).   Whenever any marker \( r^e_i \) has been targeted for \( A \) but delayed by negative requirements, we act to reserve another interval for \( \mathcal{R}_e \).  If  \( \mathcal{R}_e \) already has \( j -1  \) intervals we select \( l^e_j \) to be the first element larger than every previously defined interval for any requirement.  We then select \( h^e_j \) to be the least number currently occupied by some marker \( m_k \) such that \( 2(l^e_j + 2^{e}+1) < h^e_j  \) and place \( r^e_j \) on \( h^e_j \).  This has the effect of guaranteeing that there are at least \( 2^{e} \) elements in \( \left[ l^e_j, h^e_j \right] \) that will not be enumerated into \( \hat{C} \), leaving  \( \mathcal{R}_e \) complete control over placing these elements into \( C \).      

If at any point we observe that the first clause in \eqref{eq:properly-Q_n-satisfied} has been satisfied, we mark \( \mathcal{R}_e \) satisfied and take no more action on its behalf.  Otherwise, we act only if some \( r^e_j \) occupies \( x \) and either \( x \) is not currently targeted to enter \( A \) but \( x \in Y^{2n-1}_{e,s}  \) or \( x \) is currently targeted to enter \( A \) but \( x \not\in Y^{2n-1}_{e,s} \).  In the former case, we target \( x \) for entry into \( A \) and in the latter case, we cancel our targeting of \( x \) for \( A \) (placing \( x \) into those \( X^{2n}_{e} \) being built below \( \mathcal{R}_e \) ).  If at some stage \( s \), element \( r^e_{i,s} \) is targeted for \( A \) by a higher priority \( \mathcal{M}_{e'} \), then set \( r^e_{i,s+1} \) to the largest \( x < r^e_{i,s} \) (which must be in the reserved interval) with \( x \not\in C_s \) and enumerate \( x \) into \( C \).

We argue that each \( \mathcal{R}_e   \) only reserves finitely many elements and is eventually satisfied.  Note that we only move \( r^e_i \) at stage \( s \) if some element \( y > r^e_{i,s} \) increases its approximate \( C \)-complementing \( (e+1) \)-state above that of \( r^e_{i,s} \).  By enumerating \( r^e_{i,s+1} \) into \( C \), we cause the marker \( m_k \) occupying the least \( y' > r^e_{i,s+1} \) to be shifted down to \( r^e_{i,s+1} \).  By our remark in the discussion of \( \mathcal{M}_e \),  we know that \( \epsilon^{C}_{e+1}(r^e_{i,s+1}, s+1 ) \geq_L \epsilon^{C}_{e+1}(y, s ) \) when we move \( r^e_i \), as the new location of \( r^e_i \) must have already entered any c.e.\ set contributing to \( \epsilon^{C}_{e+1}(y, s ) \).  Combining these inequalities, we see that  \( \epsilon^{C}_{e+1}(r^e_{i,s+1}, s+1 ) >_L \epsilon^{C}_{e+1}(r^e_{i,s}, s ) \).  As there are only \( 2^e \) many \( C \)-complementing \( (e+1)  \)-states, we can move \( r^e_i \) at most \( 2^e -1 \) times. By choice of \( h^e_i \), we know that each time we can find some element in \( [l^e_i, h^e_i] \) not yet in \( C_s \).  Hence, \( r^e_i \) eventually occupies a location that is not stolen by a higher priority majorness requirement.  Now, if \( \mathcal{R}_e  \) only reserves finitely many intervals, it is satisfied, so assume it reserves infinitely many intervals.  In this case, let \( [l^e_i, h^e_i] \) be an interval with \( l^e_i \) so large that no element in this interval is permanently restrained by any \( \mathcal{N}_{e'}\) for \( e' \leq e \), and let \( x \) be the location \( r^e_i \) settles upon.  But now \( x \) will eventually settle down into a victory against \( Y^{2n-1}_e \) as described above and no more intervals will be reserved for \( \mathcal{R}_e \).


This completes the construction. We now need to verify that  it has the claimed properties.  If \( \dom \varphi_e \) is infinite then eventually every element  not in \( A \) must settle down into \( X^{2n}_e \) or one of the finitely many negative requirements  \( \mathcal{N}_{e'} \) with \( e' < e \) and \( \dom \varphi_e \) finite.  Thus, after joining these finitely many trapped balls to the first odd component of \( X^{2n}_e \) they have not yet entered we can assume that \( X^{2n}_e \) contains \( \overline{A} \).  Moreover we can make this finite adjustment without disrupting the property that balls enter the earlier components of the \( 2n \)-c.e.\ set before the later ones and only enter \( X^{2n}_{e_2} \) after \( C \).   Conversely, \( X^{2n}_e \) is contained in \( \overline{A} \) union the set of elements placed into \( A \) by requirements of the form \( \mathcal{R}_{e'} \) or  \( \mathcal{M}_{e'} \) for \( e' < e \) which in the former case is a finite set and the latter a computable set.  Thus, \( X^{2n}_e \) is contained in \( \overline{A} \) union a computable subset \( K \) of \( A \) so we can fix \( X^{2n}_e \) to be equal to \( \overline{A} \) by simply intersecting \( \overline{K} \) with every positive (odd) component of \( X^{2n}_e \).  Since we do not alter the negative components of \( X^{2n}_e \), we do not slow down any elements from leaving \( X^{2n}_e \), and so retain the required tardiness property.  But, by taking elements out of the odd components but not the even ones we may now violate \eqref{E:A-and-C:nested} and \eqref{E:A-and-C:nicely-nested}.  

However, the need to adjust \( X^{2n}_e \) after the construction is really only a consequence of our decision to cast the construction as a pinball argument for ease of presentation rather than a \( \Pi^{0}_2 \) tree construction.  By performing this construction in the same fashion as that in \ref{T:3tardy}, our ad hoc modification of  \( X^{2n}_e \) becomes unnecessary as negative requirements \( \gamma \) can simply delay adding balls to the components of \( X^{2n}_e \) until every higher priority requirement \( \mathcal{R}_e \) that \( \gamma \) guesses will act infinitely often believes it will not need to add that ball to \( X^{2n}_e \).  Whenever a requirement \( \mathcal{R}_e \) that \( \gamma \) believes only needs to act finitely many times act \( \gamma \) can simply reset its construction of \( X^{2n}_e \) and begin from scratch.  Understood in terms of the tree construction the reservation of balls by \( \mathcal{R}_e \) acting at \( \gamma \) simply becomes the constraint that any nodes above or to the right of \( \gamma \) cannot pick these elements as new balls.  

\end{proof}

We can prove similarly a version of Lemma \ref{lem:major-2n-tardy} for the odd case.

\begin{lemma}\label{lem:major-2n+1-tardy}
	For every \( n \geq 1 \), there is a properly \( 2n+1 \)-tardy set \( A \), a c.e.\ set $Z$ disjoint from $A$, so that \( \hat{A}= A \sqcup Z \)  satisfies the conditions of Lemma \ref{lem:major-2n-tardy}.  So in particular, \( \hat{A} \) is \( 2n \)-tardy and satisfies all the demands on the enumeration order.
\end{lemma}
\begin{proof}
	One could take a c.e.\ set \( \hat{A} \)  built to satisfy Lemma \ref{lem:major-2n-tardy} and split up the elements we enumerate into \( \hat{A} \) into exactly one of the bins \( A \) or \( Z \).  However, since we must also ensure that \( A \) is \textit{properly} \( 2n+1 \)-tardy it is preferable to dynamically build \( \hat{A} \) as above and only once we have made an irrevocable commitment to place \( x \) into \( \hat{A} \) do we decide whether to put \( x \) into \( A \) or \(  Z \).  If we put \( x \) into \( \hat{A} \) as the result of \( \mathcal{M}_e \), we place \( x \) into \( {A} \).  Then,  \( \mathcal{M}_e \) will be satisfied in the same manner as before.   We now focus on those balls  used by some requirement \( \mathcal{R}_e \).   Without loss of generality, we may assume these balls are not stolen by some higher priority \( \mathcal{M}_{e'} \) since we can always react to that event by handing out a new ball to \( \mathcal{M}_e \). As we argued above,  eventually \( \mathcal{M}_e \) will receive a ball that does not get stolen.

We also place into \( A \) any balls that were enumerated into \( \hat{A} \) by some \( \mathcal{R}_e \)  before entering the \( 2n \)-th component  \( X^{2n}_{e_{2n}} \) of any node.  These balls entered \( \hat{A} \)  to obtain an immediate victory by showing that either \( A \neq B \) or that \( x \) does not leave \( Y^{2n}_{e,s} \) soon enough before entering \( B \).  By placing these balls into \( A \), we ensure that either \( Y^{2n}_{e,s} \)  does not witness that \( B \) is \( 2n \)-tardy or \( A\neq B \).  This leaves only the case where \( x \) enters \( Y^{2n}_{e_{2n-1},s} \) and our construction of \( \hat{A} \) responds by starting \( x \) down towards the \( \hat{A} \) by placing it in the sets \( X^{2n}_{e_{2n}} \).  As far as the construction of \( \hat{A} \) is concerned, once \( x \) has entered \( X^{2n}_{e_{2n}} \), it must continue on into \( \hat{A} \) (modulo possible finite injury).  However, when our \( x \) reaches the root node we are not forced to place it into \( A \).  Instead, we check if \( x \) is still in \( Y^{2n}_{e,s} \).  If so, we place \( x \) into \( A \) for the immediate victory.  Alternatively if \( x \not\in Y^{2n}_{e,s}  \), i.e.,   \( x \in Y^{2n}_{e_{2n},s}  \), then either \( Y^{2n}_{e} \) fails to witness that \( B \) is \( 2n \)-tardy or \( x  \) enters \( B \) so we place \( x \) into \( Z \) rendering \( A = \hat{A} - Z \) properly \( 2n+1 \)-tardy.  

Since the only elements entering \( \hat{A} \) but not \( A \) pass through all the intermediate components \( X^{2n}_{e_{k}} \) in order at the negative requirements below \( \mathcal{R}_e \), the ordering properties trivially hold at these nodes.  At the remaining nodes, \( x \) may have become stuck in \( X^{2n}_{j_{2}} \) or some other component.  However, this concern is easily addressed by taking any balls we place into \( Z \) and slowly running them through the components of \( X^{2n}_{j} \) in order. Then, we can use a slower enumeration of \( Z \) to add a final component to \( X^{2n}_{j} \), making a \( 2n+1 \)-c.e.\ set that satisfies the hypotheses of the theorem.
\end{proof}

We now prove Theorem \ref{T:Q_nsatisfied}, i.e., we show that the above sets $A$ in Lemmas \ref {lem:major-2n-tardy}
and  \ref{lem:major-2n+1-tardy} satisfy \( Q_m(A) \).  Since the definition of \(Q_{2n+1}(A)\) is simply 
\mbox{\((\exists Z \subseteq \bar{A})Q_{2n}(A \cup Z)\)} and \( \hat{A}= A \sqcup Z \) satisfies the conditions of Lemma \ref{lem:major-2n-tardy}, we simply need to show that \(Q_{2n} (A)\) holds for $A$ as in Lemma \ref{lem:major-2n-tardy}.

\subsection{Verifying Satisfaction}

To show that \( A \) satisfies \( Q_{2n}(A) \) via the \( C \) built above we fix some arbitrary \( \vec{B} = (B_1, \ldots, B_n) \) and will construct \( \vec{D}=(D_1, \ldots, D_n) \) in response.  Furthermore for every \( S_j \sqsubset C \) we must also  describe the sets \( \vec{T^j}=(T^j_1, \ldots, T^j_n) \) we will play in response.  To gain better control over our construction we fix an effective enumeration  \( \{(S_j, \hat{S_j}) \mid j \in \omega \}  \) containing all disjoint pairs of c.e.\ subsets of \( C \) requiring, by way of the Slowdown Lemma \cite{SoareBook} that the indices we list satisfy \( S_j \cup \hat{S_j} = C \searrow (S_j \cup \hat{S_j}) \) in the canonical stagewise enumeration of c.e.\ sets.

We use the tardiness of \( A \) to force elements in \( (A\cup T_{i+1}) \cap(S_j\cap T_i) \) (where \( T_{n+1} = \emptyset \)) into  \( B_i \) whenever \eqref{E:Q_2nif} is satisfied.  Since \eqref{E:Q_2nif} forces \( B_i \) to copy \( D_i \) on \( S_j - A \), this occurs automatically for \( x \in T_{i+1} - A \).  To deal with \( x \in A \) we note that if \eqref{E:Q_2nif} is satisfied we can computably measure how long it takes elements in \( D_i \) and \( S_j \) to appear in either \( A \) or \( B_i \).  In this case, let \( p_j(s) \) be the amount of time it takes for elements entering \( C \) before stage \( s \) that will eventually enter \( S_j \cap D_i \) to enter \( B_i \) or \( A \).  By tardiness, there exists \( X_e^{2n} = \overline{A} \) working faster than this delay $p_j$.  
The trick is to then use  \( X_e^{2n}\) to build \( D_i \) as a version of \( X_{e_{2i}} \) and \( T^{j}_{i} \) as a version of \( X_{e_{2i-1}} \) so that any \( x \) in \( A \cap S_j \cap T_i \) has to first enter either \( B_i \) or \( A \) before entering \( A \) and therefore must enter \( B_i \).

To show \( Q_{2n}(A) \) holds, we use the \( 2n \)-c.e.\ sets \( X^{2n}_{j} \) demonstrating that \( A \) is \( 2n \)-tardy with respect to $p_j$ to construct \( \vec{D} \) and \( \vec{T} \).  The key point  is that, for those elements we care about, \( D_i \) will behave like a modified version of some \( X^{2n}_{e_{2i}} \) and \( T_i \) will behave like a modified \( X^{2n}_{e_{2i-1}} \).  Thus, \( \vec{T} \) can be thought of as elements that may stay out of \( A \) and \( \vec{D} \) as elements that may enter \( A \).  Now, pretending we get to play both \( \vec{T} \) and \( \vec{D} \) in response to the $n$-c.e.\ set  $X_e^{2n}$ witnessing tardiness with respect to $p_j$, we would proceed as follows.

 
To overcome our inability to build \( \vec{D} \) in response to the choice of \(S_ j\cup \hat{S}_j \) we aim to somehow split up the construction of \( D_i \) so that on \( T^j_1 \cap S_j \) the construction responded to \( S_j \).  Unfortunately these sets are not disjoint but the approach remains valid as they make compatible demands.   Unfortunately this strategy only works when \( \vec{D}, \vec{B}, S_j \) really satisfy \eqref{E:Q_2nif} and and even then we must locate the correct \( 2n \)-c.e.\ set \( X^{2n}_e \) with respect to $p_j$.  We manage this complexity using a \( \Pi^0_2 \) guessing procedure at the true path \( f(k) \) described below.  Recall that $A$ and $C$ are fixed from  Lemma \ref{lem:major-2n-tardy}, and \( \vec{B} \) is fixed and arbitrary.  We will formally define $p_{\alpha}$ later.  
 
	\begin{align*}
		f(0) &= j_0 \text{ where } C = W_{j_0} \\
				f(2k+1) &= (\mu e)(X^{2n}_e \text{ satisfies Lemma \ref{lem:major-2n-tardy} with respect to } p_{f\restriction{2k+1}} ) 
\\
		f(2k+2) &= (\mu j > f(2k))\left[  (S_j \sqcup \hat{S_j}) = C \land \text{ \eqref{E:Q_2nif} holds for \( \vec{D} \).} \right] 
	\end{align*}

We adopt the convention that \( \alpha, \beta, \gamma \) only range over strings of even length, \( X^{\alpha} \) refers to \( X^{2n}_{e_n} \), \( S_\alpha \) to \( S_{j_n} \) and denote the sets \( \vec{T} \) built in response to \( S_\alpha \) by \( \vec{T^{\alpha}} \) for \( \alpha = (j_0,e_0, j_1, e_1, \ldots, j_n, e_n) \).  We adopt a standard \( \Pi^0_2 \) approximation argument so that \( \alpha \subset f \) iff \( \alpha \) is the \( <_L \) string of that length satisfying \( (\exists^{\infty} s)\ [ f_s \supseteq \alpha] \).   For every \( x \), we keep track of \( \Gamma(x,s) \), the leftmost substring of \( f_t \) of length \( x \) for \( t \in [x,s] \) observing that if \( \alpha \subseteq f \) then for all but finitely many \( x \) we have \( \Gamma(x,\infty) \supseteq \alpha \).
Before we specify the construction of the function \( p_\alpha \), we detail the construction of \( \vec{D} \) and \( \vec{T^{\alpha}} \).

\begin{equation*}
	\begin{split}
		x \in T^\alpha_{1,s+1} \leftrightarrow & x \in T^\alpha_{1,s} \\
			& \lor  \left[ s+1 \geq x   \land  \Gamma(x,s) \supseteq \alpha  \land  x \in X^{\alpha}_{1,s} - C_s \right]
	\end{split}	 
\end{equation*}

\begin{subequations}
\begin{align}\label{E:def-D-and-Ti}
	s_x &= (\mu t)(x \in C_t) \\	
	\begin{split}
	\alpha_x &= (\mu \beta \subseteq \Gamma(x,s_x))(\exists t)\\
					 &\left[ [x \in T^{\beta}_{1,s_x} \cap S_{\beta,t}]  \land (\forall \gamma\subset\beta)[x\not\in T_{1, s_x}^\gamma  \lor x\in\hat{S}_{\gamma, t}]  \right] 
	\end{split}\\
	x & \in T^{\alpha}_{i+1} \leftrightarrow x \in X^{\alpha_x}_{2i+1,s} \cap T^{\alpha}_{1,s}\label{E:def-D-and-Ti:Ti}\\
	x & \in D_{i}  \leftrightarrow x \in X^{\alpha_x}_{2i}\label{E:def-D-and-Ti:Di}
\end{align}  
\end{subequations}

Note that if \( x \) is not in \( T^{\alpha}_{1,s_x} \),  it is not in \( T^{\alpha}_1 \).  Additionally, given \( x \in C_s \) for \( \alpha \) on the true path, it is a computable question whether \( \alpha_x \subseteq \alpha \) since $S_j\cup \hat{S}_j$ is actually a split of $C$ for $\beta\subseteq\alpha$.  We now define \( p_\alpha(s) \).

For every \( x, s \) define \( p_\alpha(i,x,s)=s \) if \( x \not\in C_s\cup X^{\alpha_x}_{2i,s+1} \lor \alpha_x \nsubseteq \alpha  \) otherwise set

\begin{align*}
	p_\alpha(i,x,s) &= (\mu s' \geq s)(\alpha^{-}\subseteq f_s \land x \in B_{k,s'} \cup A_{s'} )\\
\shortintertext{Then define}
	p_\alpha(s) &= 1+\max_{\substack{x \leq s \\ i \leq n}} p_\alpha(i,x,s)  
\end{align*}

\begin{lemma}\label{lem:delay-is-good}
	If \( \alpha^{-} \) is on the true path then \( p_{\alpha}(s)= p_{\alpha^{-}}(s)\) is a total function.
\end{lemma}
\begin{proof}  It suffices to  show that $p_{\alpha}(i, x, s)$ is defined for all $x\le s$ and $i\le n$.  If $x \not\in C_s\cup X^{\alpha_x}_{2i,s+1} \lor \alpha_x \nsubseteq \alpha$, then this is clear.  
By definition of \( D_i \) and \( \alpha_x \), if \( x \in X^{\alpha_x}_{2i} \) then \( x \in D_i \cap S_{\alpha_x} \).  Thus, if \( p_\alpha(i,x,s) \neq s \) then as \( \alpha^{-} \) is on the true path \eqref{E:Q_2nif} holds so \( x \in A \) or \( x \in B_i \).  Hence,  $p_{\alpha}(i, x, s)$ is defined by the second clause in its definition.  
\end{proof}

Note the importance of defining \( p_{\alpha} \) using only \( \alpha^{-} \) is to avoid any circularity in selecting the \( 2n \)-c.e.\ set witnessing tardiness for \( p_{\alpha} \) at node \( \alpha \).

\begin{lemma}\label{L:alpha_x}
If $\alpha$ is on the true path, then for all but finitely many $x$, 
$$x\in T_1^\alpha\cap S_\alpha \ \rightarrow\  \alpha_x\subseteq \alpha.$$
\end{lemma}
\begin{proof}
For any $x\in T_1^\alpha$, there exists an $s$ such that $\Gamma(x, s)\supseteq\alpha$ with \mbox{$x\le s\le s_x$.}   Let $s'$ be a stage such that for all $t\ge s'$, $f_s$ is not left of $\alpha$.  Then, for each $x>s'$, we have $s_x\ge s\ge x>s'$ since $x\not\in C_{s'}$ by properties of the enumeration of c.e.\ sets.  Hence, for $x>s'$, we have that $\Gamma(x, s_x)$ (and then $\alpha_x$) is not to the left of $\alpha$.  Fix $x>s'$ such that $x\in T_1^\alpha\cap S_\alpha$.  Since $\Gamma(x, s)\supseteq\alpha$ for $s_x\ge s\ge x>s'$ and $\alpha_x$ is not to the left of $\alpha$, we have that $\alpha_x$ must either extend $\alpha$ or be a substring of $\alpha$.  Since $x\in T_1^\alpha\cap S_\alpha$, we have that $\alpha_x\not\supset \alpha$ by the second conjunct in the definition of $\alpha_x$.   
Thus, $\alpha_x\subseteq\alpha$  for all $x>s'$ such that $x\in T_1^\alpha\cap S_\alpha$ as desired.

\end{proof}
\begin{lemma}\label{lem:delay-into-A}
	If \( \alpha \) is on the true path, then \( A \cap S_\alpha \cap T^{\alpha}_i  \subseteq^* B_i \cap S_\alpha \cap T^{\alpha}_i \) 
\end{lemma}
\begin{proof}
By Lemma \ref{L:alpha_x}, for all but finitely many $x\in T_1^\alpha\cap S_\alpha$, we have that $\alpha_x\subseteq\alpha$.   Take such an $x$ in $ A \cap S_\alpha \cap T^{\alpha}_i$ (Note that $ A \cap S_\alpha \cap T^{\alpha}_i\subseteq T_1^\alpha\cap S_\alpha $.) By definition, \( x \in S_{\alpha_x} \cap T^{\alpha_x}_{1} \).  Since \( x \in T^{\alpha}_i \), it follows that \( x \in X^{\alpha_x}_{2i-1} \), and since \( x \in A \), we also have \( x \in X^{\alpha_x}_{2i} \).  By \eqref{E:A-and-C:nicely-nested}, if \( s+1 \) is the least stage such that \( x \in X^{\alpha_x}_{2i, s+1} \) then \( x \in X^{\alpha_x}_{2i-1, s} \).  Furthermore, by \eqref{E:A-and-C:C-first} we must have \( x \in C_s \), so \( p_{\alpha_x}(i,x,s) \) is defined by way of the second clause.
	
	By lemma \ref{lem:delay-is-good}, we know that \( t= p_{\alpha_x}(i,x,s) \) is well defined hence either \( x \in A_t \) or \( x \in B_t \).  However, \( p_{\alpha_x}(s) \geq p_{\alpha_x}(i,x,s) \) and \( x \in X^{\alpha_x}_{2i-1, s} \) so \( x \not\in A_t \).  Hence \( x \in B \).
\end{proof}

\begin{lemma}\label{lem:B-contains}
	If \( \alpha \) is on the true path then 
	\begin{equation*}
		\left[\begin{alignedat}{3}
	 (A\cup T^{\alpha}_2)&\cap(S_\alpha\cap T^{\alpha}_1) &&\subseteq^* & B_1 &\cap(S_\alpha\cap T^{\alpha}_1) \\
	 (A\cup T^{\alpha}_3)&\cap(S_\alpha\cap T^{\alpha}_2) &&\subseteq^* & B_2 & \cap(S_\alpha\cap T^{\alpha}_2)\\
	 & &&\ldots \\
	 (A\cup T^{\alpha}_n)&\cap(S_\alpha\cap T^{\alpha}_{n-1}) &&\subseteq^* & B_{n-1} &\cap(S_\alpha\cap T^{\alpha}_{n-1}) \\
	 A&\cap(S_\alpha\cap T^{\alpha}_n) &&\subseteq^* & B_n &\cap(S_\alpha\cap T^{\alpha}_n) 
	 \end{alignedat}\right]
	\end{equation*}
\end{lemma}
\begin{proof}
	By Lemma \ref{lem:delay-into-A}, the last clause is established. Hence,  let \( x \) be in  \mbox{\( (A\cup T^{\alpha}_{i+1})\cap(S_\alpha\cap T^{\alpha}_i) \)} for \( i < n-1 \) and such that $\alpha_x\subseteq \alpha$ (This is true for all but finitely many of these $x$ by Lemma \ref{L:alpha_x}.)  Again by lemma \ref{lem:delay-into-A} the result is established except for \( x \in T^{\alpha}_{i+1} - A \) so assume this is the case.  Since \( x \in T^{\alpha}_{i+1} \) it follows that \( x \in X^{\alpha_x}_{2i+1} \) and by \eqref{E:A-and-C:nested} we have \( x \in X^{\alpha_x}_{2i} \) so \( x \in D_i \).  But as \( \alpha \) on the true path \( B_i \cap (S_\alpha - A) = D_i \cap (S_\alpha - A) \) and \( x \in D_i \cap (S_\alpha - A) \).  Hence \( x \in B_i \) completing the proof.
\end{proof}

\begin{lemma}\label{lem:B-contained-in}
	If \( \alpha \) is on the true path then 
	\begin{equation*}
		\left[\begin{alignedat}{3}
	 (A\cup T^{\alpha}_2)&\cap(S_\alpha\cap T^{\alpha}_1) &&\supseteq^* & B_1 &\cap(S_\alpha\cap T^{\alpha}_1) \\
	 (A\cup T^{\alpha}_3)&\cap(S_\alpha\cap T^{\alpha}_2) &&\supseteq^* & B_2 & \cap(S_\alpha\cap T^{\alpha}_2)\\
	 & &&\ldots \\
	 (A\cup T^{\alpha}_n)&\cap(S_\alpha\cap T^{\alpha}_{n-1}) &&\supseteq^* & B_{n-1} &\cap(S_\alpha\cap T^{\alpha}_{n-1}) \\
	 A&\cap(S_\alpha\cap T^{\alpha}_n) &&\supseteq^* & B_n &\cap(S_\alpha\cap T^{\alpha}_n) 
	 \end{alignedat}\right]
	\end{equation*}
\end{lemma}

\begin{proof}

	

Assume \( x \in B_i \cap(S_\alpha\cap T^{\alpha}_i) \), and suppose that $\alpha_x\subseteq\alpha$ (This holds for all but finitely many \( x \in B_i \cap(S_\alpha\cap T^{\alpha}_i) \) by Lemma \ref{L:alpha_x}).  
If \( x \in A \), we are done,  so suppose not.  Since \eqref{E:Q_2nif} is satisfied and \( x \in B_i \cap (S - A) \), we must have \( x \in D_i \).  Thus, by \eqref{E:def-D-and-Ti:Di} we have \( x \in X^{\alpha_x}_{2i} \).  If  \( i = n \), we have that  \( x \in X^{\alpha_x}_{2n}\subset A \) since $\alpha_x\subseteq \alpha$ is on the true path.  This contradicts our original assumption.    For $i<n$,  since \( x \not\in A \), we also have  that \( x \in X^{\alpha_x}_{2i+1} \) by nesting.   Since \( x \in T^{\alpha}_{i}\subset T^{\alpha}_{1} \),  \eqref{E:def-D-and-Ti:Ti} entails that \( x \in T^{\alpha}_{i+1} \).  Hence, \( x \in (A\cup T^{\alpha}_{i+1})\ \cap\ (S_\alpha\cap T^{\alpha}_{i}) \), completing the lemma.
\end{proof}

We now finish demonstrating that \( A \) satisfies \( Q_{2n}(A) \).  Given any \( \vec{B} \), we respond by building \( \vec{D} \) as above.  Given \( S_j \sqsubset C \), we first check whether \( \vec{B}, \vec{D}, S_j \) satisfy \eqref{E:Q_2nif}, and if not, we are done.  If so, there is some finite sequence \( \alpha = (j_0,e_0, j_1, e_1, \ldots, j_n, e_n) \) along the true path with \( j_n = j \), and Lemmas \ref{lem:B-contains} and \ref{lem:B-contained-in} above guarantee that \( S_j, \vec{T^{\alpha}}, \vec{B} \) satisfy \eqref{E:Q_2nthen} up to $=^*$.  By construction and Lemma \ref{L:alpha_x}, \( D_{i+1} \subseteq^* D_i \) and \( T_{i+1} \subseteq^* T_i \).   To ensure that \eqref{E:Q_2nthen}, \( D_{i+1} \subseteq D_i \), and \( T_{i+1} \subseteq T_i \)  hold exactly, remove the elements in $S_\alpha\cap T_1^\alpha$ that violate the equalities in \eqref{E:Q_2nthen} or the subset properties from each  $T_i^\alpha$.  
Only \( T_1 \supseteq \overline{C} \) remains potentially unsatisfied.  Note that all elements just removed from $T_i^\alpha$   are elements of $C$, so these finite modifications do not affect \( T^{\alpha}_1 \supseteq \overline{C} \).

If \( \alpha \supseteq f \), we have \( \Gamma(x, \infty) \supseteq \alpha \) for almost every \( x \) and \( X^{\alpha}_1 \supseteq \overline{C} \).  Hence, \( T^{\alpha}_1 \supseteq^{*} \overline{C} \).  Thus, a finite modification of \( T^{\alpha}_1 \) suffices to ensure \( T^{\alpha}_1 \supseteq \overline{C} \).  Since \eqref{E:Q_2nif} only makes demands on \( S_\alpha \subseteq C \), these modifications do not affect the satisfaction of \eqref{E:Q_2nif}.      This completes the proof of satisfaction.

\section{A \texorpdfstring{low$_2$}{low2} and simple very tardy}\label{S:low2simple}

Previously Harrington and Soare established  the following theorem in \cite{the-delta-0-3-automorphism-method}.

\begin{theorem}[Harrington and Soare]\label{thm:low-simple-imply-almost-prompt}
	If \( A \) is low (or even semi-low) and simple then \( A \) is almost prompt.
\end{theorem}

Since a very tardy set is simply one that is not almost prompt,  this theorem shows  that no very tardy can be both low and simple.  Harrington and Soare's proof demonstrates that if \( A \)  is very tardy and semi-low there is a computable function that grows fast enough so that the corresponding \( n \)-c.e.\ complement of \( A \) is forced to leave an infinite c.e.\ set in \( \overline{A} \).  We show that Theorem \ref{thm:low-simple-imply-almost-prompt} cannot be improved by constructing an example of a low$_2$ simple very tardy set.  This example provides a negative answer to Question \ref{Q:Q1HS} of Harrington and Soare.   We first offer a sketch of the tension in Theorem \ref{thm:low-simple-imply-almost-prompt} so as to motivate the construction of our example.

Building a low set requires that we eventually preserve computations of the form \( \recfnl{e}{A_s}{e} \), while simplicity requires that if \( \REset[][s]{i} \) continues to grow, we eventually enumerate one of its members into \( A \).  Normally, we  build a low simple set by only allowing a finite number of computations \( \recfnl{e}{A_s}{e} \) to restrain  elements we see enter \( \REset{i} \) out of  \( A \).  However, building a very tardy set requires that we announce our intention to enumerate some element \( y \) into \( A \) long in advance. During the intervening time, a computation \( \recfnl{e}{A_s}{x} \) might converge and impose a restraint that  \( y \) is obligated to respect. Hence,  \( y \) must abandon its previously announced intention to enter \( A \).  If \( A \) was meant to be \( 2 \)-tardy, this alone would cause a failure since \( 2 \)-tardy sets cannot revoke their announced intentions to place elements into \( A \). It might seem, on the other hand, that if we only aim to build a very tardy we could simply choose to leave  \( y \) out of \( A \) and wait for another chance to place an element from \( \REset{i} \) into \( A \).  However, by the time we observe that some \( y_n \) enters \( \REset{i} \), some later \( \REset{i_n} \) may have already attempted to enumerate \( y_n \) into \( A \) and abandoned that attempt  in response to a restraint from some computation  \( \recfnl{e_n}{A_s}{e_n} \). Indeed, each \( y_n \) entering \( \REset{i} \) may have already exhausted its guesses about entering \( A \) so  that \( \REset{i} \) no longer has the opportunity to place \( y \) into \( A \). 

It is clear from the above discussion that  the need to restrain elements from entering \( A \) creates the potential to `use up' the elements of some infinite c.e.\ set before we have the chance to place one of its members into \( A \).  Since lowness requirements in general require imposing some kind of restraint, it is interesting to see that Harrington and Soare's result fails for a weaker notion of lowness.

\begin{theorem}\label{thm:low2-2tardy-simple}
	There is a simple \( 2 \)-tardy set \( A \) with \( \jjump{A}=\zerojj \)
\end{theorem}

For the sake of readability, we denote the \( e \)th partial computable function  \( \recfnl{e}{}{} \) as \( p_e \) when regarded as a potential total computable function \( p \) as in  Definition \ref{D:verytardy}.
	
Our construction will satisfy a version of the following two requirements.

	\begin{align}
	 	\tag*{ $\mathcal{N}_{e}$ }  p_e \text{ is partial  or } \text{there is a 2-c.e.\ set } X^2_e = \overline{A} \text{ s.t. } X^2_{e,s} \isect A_{p(s)} = \eset  \\
	 	\tag*{ $\mathcal{P}_{e}$ } \card{\REset{e}} = \infty \implies A \isect \REset{e} \neq \eset  \\
	 	\tag*{ $\mathcal{R}_{e}$ } \text{If } \limsup_{s} \card{\REset(A_s)[s]{e}} = \infty \implies \card{\REset(A){e}} = \infty  
	\end{align}

Requirement \( \mathcal{P}_e \) guarantees that \( A \) is simple, and \( \mathcal{N}_e \) ensures that \( A \) is \( 2 \)-tardy.  Notice that 
\[
\limsup_{s} \card{\REset(A)[s]{e}} = \infty  \iff \card{\set{\pair{s}{n}}{\ s =\murec{t}{\card{\REset(A)[t]{e}} \geq n}}} = \infty
\]
This fact and the fact that \( \set{e} {\ \card{\REset(A){e}} = \infty } \Tequiv \jjump{A} \) guarantee that \( \jjump{A}=\zerojj \).  Our actual construction will take place on a tree.  Taking advantage of this tree structure, however,  will require a minor modification of \( \mathcal{R}_e \).

\subsection{Tree Argument}

For every element \( \alpha \in \bstrs \), we  tentatively associate a module \( \mathcal{Q}_\alpha \) that tries to implement an associated requirement according to the following rule.

\begin{equation*}
	\mathcal{Q}_\alpha \text{ implements } \begin{cases}
																	\mathcal{N}_e  & \text{ if } \lh{\alpha} = 3e \\
																	\mathcal{P}_e  & \text{ if } \lh{\alpha} = 3e+1 \\
																	\mathcal{R}_e  & \text{ if } \lh{\alpha} = 3e+2 \\
															\end{cases}
\end{equation*} 

In the interest of clarity, we simply write \( \mathcal{N}_\alpha, \mathcal{P}_\alpha \) or \( \mathcal{R}_\alpha \) to refer to the module \( \mathcal{Q}_\alpha \) in the case that it is associated with \( \mathcal{N}_e, \mathcal{P}_e \) or \( \mathcal{R}_e \), respectively.  Conversely, we use \( \mathcal{Q}_{n} \) to denote the requirement implemented by the modules \( \mathcal{Q}_\alpha \) with \( \lh{\alpha}=n \).  As a notational convenience, we  write \( p_\alpha\) and \( \REset{\alpha} \) in place of \( p_e\) and \( \REset{e} \) when discussing \( \mathcal{N}_\alpha, \mathcal{P}_\alpha \) and \( \mathcal{R}_\alpha \).  We say \( \mathcal{Q}_\alpha \) is satisfied to indicate that \( \mathcal{Q}_\alpha \) satisfies \( \mathcal{Q}_{\lh{\alpha}} \).

We define the true path function, \( f\map{\omega}{\set{0,1}{}} \) with the property that if \( \alpha \subseteq f \) then the module \( \mathcal{Q}_\alpha \)  satisfies \( \mathcal{Q}_{\lh{\alpha}} \).  In an abuse of notation, we also view \( f \) as a function from \( \functo{f}{\bstrs}{\set{0,1}{}} \) where \( f(\alpha) \) indicates the manner in which \( \mathcal{Q}_{\lh{\alpha}} \) is satisfied so that    \( f(\lh{\alpha}+1) = f(\lh{\alpha})\concat f(\alpha) \).  

At any given stage \( s \), we will have an approximation \( f_s \in \bstrs \) to \( f \) where we similarly abuse notation and write \( f_s(\alpha) \) to denote \( f_s(\lh{\alpha}) \) provided \( f_s \supseteq \alpha \).  At a given stage \( s \), only those modules lying along the approximation \( f_s \) receive attention (are visited) so in what follows we always assume that \( \alpha \subseteq f_s \) when describing \( \mathcal{Q}_\alpha \).  If \( \alpha \) is not visited at stage \( s \), then we stipulate that \( f_s(\alpha) \) is undefined.  Thus, if \( \alpha \) is the \( \subset \)-maximal node visited at \( s \), then \( f_s=\alpha\concat f_s(\alpha) \).   The construction will ensure that \( f= \liminf_s f_s \) and the action of each module will be described uniformly in \( \alpha \) so that \( f_s \in \bstrs \) is a computable function of \( s \).

\subsubsection{Motion On The Tree}

Various numbers, called balls, will be located on our nodes.  The function  \( \curnode{x}{s} \) equals the node occupied by \( x \) at stage \( s \) or \( \diverge \) if \( x \) is not on the tree.  One should think of the tree as growing upward with elements trickling down the tree towards \( A \) during the course of the construction.   Any element reaching \( \estr \) is immediately placed into \( A \), so \( \curnode{x}{s} = \estr \) if and only if \( x \in A_s \) after which point \( \curnode{x}{s} \) it can no longer change value.    

At a given stage \( s \), the construction will begin by executing \( \mathcal{Q}_{\estr} \).  Suppose the module \( \mathcal{Q}_\beta \) receives attention and sets \( f_s(\beta) =i \).  If,  for all \( \lh{\beta} < s \), requirement \( \mathcal{Q}_\beta \) did not enumerate a ball onto the tree and no ball occupies any \( \beta' \subseteq \beta\, \concat\, i \), then \( \mathcal{Q}_{\beta\ \concat\, i} \) receives attention.  This condition ensures that whenever \( \mathcal{Q}_\alpha \) is executed, every ball below \( \alpha \) on the tree is already in \( A \).   We say that the node \( \beta \) is visited at stage \( s \) if the module \( \mathcal{Q}_\beta \) is executed at stage \( s \).  If \( \beta \) is visited at stage \( s \) and \( f_s(\beta)=0 \), we say that \( \beta \) is expansionary at stage \( s \).  We write \( \alpha <_L \beta \) to indicate that for some \( l \), node \( \alpha\restr{l} \) occurs lexicographically before \( \beta\restr{l} \) and if \( f_s <_L \beta \), we say that \( \beta \) (equivalently \( \mathcal{Q}_\beta \)) is reset at stage \( s \).

If \( \curnode{x}{s} \) is reset at stage \( s \), then \( x \) is removed from the machine and \( \curnode{x}{s} \) is undefined.  Intuitively, this corresponds to abandoning our plans regarding \( x \) because our approximation to the true path moved to the left of \( \curnode{x}{s} \).   Since the \( \mathcal{N}_\beta \) modules are the only modules that  wish to delay the entry of elements into \( A \), we only allow balls to occupy nodes of the form \( \beta\, \concat\, 0 \) where  \( \mathcal{Q}_\beta \) implements \( \mathcal{N}_\beta \).  That is, balls flow down towards $A$ until they are restrained by some delay function \( p_\beta \).  We ensure by definition that whenever \( f_s \) extends \( \curnode{x}{s-1}=\beta\, \concat\, 0 \), the restraint imposed by \( \mathcal{N}_\beta \) is released and the ball is moved to the $\subset$-maximal node \( \nu\concat[0] \subset \beta \) with \( \mathcal{Q}_\nu \) implementing \( \mathcal{N}_\nu \) or into \( A \) if no such \( \nu \) exists.

 We write:

	\begin{align*}
		\Lstg{\alpha}{s}&=\max_{t\leq s}\ \{t \mid \alpha \text{ is visisted at stage } t\}\\
		\reset{\alpha}{s}&=\max_{t\leq s}\ \{t\mid \mathcal{Q}_\alpha \text{ is reset at stage } t\}
	\end{align*}

Any \( \mathcal{P}_\alpha\) for \( \alpha \supseteq \beta\, \concat\, 0  \) can place
elements  on the tree at an allowed node \( \beta\ \concat\, 0 \).  However, at stage \( s \),  \( \mathcal{P}_\alpha \) is only allowed to place a ball \( x \) onto the tree provided \( \reset{\alpha}{s} < x \leq s \).  Note that this implies that  there is no \( t < s \) with  \( \curnode{x}{t} <_L \alpha \).
 This ensures that, when reset, a node  starts with a fresh set of balls.  In other words,   nodes cannot recycle balls that have been to their left, or equivalently,  a node can recycle a ball only when that node has never been in a position to notice that the ball was used previously.  We adopt the convention that the use of \( \REset[A][s]{\alpha} \) is no more than \( s \).  Hence, the action of the machine ensures that if \( \alpha \) is visited at stage \( s \) then no \( \alpha' \) with \( \alpha <_L \alpha' \) can disrupt this computation because \( \alpha' \) was reset and any new balls that might be placed in for $\alpha'$ will be greater than $s$.  

The requirements \( \mathcal{R}_\alpha\) for \( \beta \subseteq \alpha \) will work together to define a restraint function \( \rstrn{\beta}{s} \), and \( \mathcal{P}_\beta \) will only be allowed to add a ball \( x \) to the machine only if \( x > \rstrn{\beta}{s} \).  More precisely, each \( \mathcal{R}_\alpha \) will have its own restraint function  \( \rstrn[\alpha]{}{} \), and we define

\begin{equation*}
	\rstrn{\beta}{s} = \max_{\alpha \subset \beta} \rstrn[\alpha]{\beta}{s} 
\end{equation*}

If neither the action of the tree nor the restraint function above bars \( \mathcal{P}_\beta \) from placing \( x \) on the tree at stage \( s \), then we say that \( x \) is available to \( \beta \) at stage \( s \).

\subsubsection{The \texorpdfstring{\(\mathcal{N}_\alpha \)}{N_\alpha} module}

If \( \mathcal{Q}_\alpha \) implements \(\mathcal{N}_\alpha \) then we define

\begin{equation*}
f(\alpha) = \begin{cases}
																						0 & \text{if } p_\alpha \text{ total} \\	
																						1 & \text{if } p_\alpha \text{ partial} 
																					\end{cases}
\end{equation*}

\begin{equation*}
	l^{\alpha}(s)=\murec{z}{p_{\alpha, s}(z)\uparrow} \\
\end{equation*}

\begin{equation*}
f_{s}(\alpha) = \begin{cases}
																						0 & \text{if } l^{\alpha}(s) > \max [t<s\mid f_t \supseteq \alpha\concat\, 0] \\
																							
																						1 & \text{otherwise}\\
									\end{cases}
\end{equation*}

Recall that if $f_s(\alpha)=0$, we call $s$ an expansionary stage.  
As long as it is not reset, \( \mathcal{N}_\alpha \) builds the set \( X^2_\alpha \) as follows.  At stages where \( f_s(\alpha)=0 \), every \( x < s \) with \( \curnode{{x}}{{s}} \nsubseteq \alpha \) and $x\not\in A_s$ is enumerated into \( X^2_{\alpha_1} \).  Also, whenever an element \( x \) is placed at \( \alpha\concat\, 0 \), it is enumerated into \( X^2_{\alpha_2} \).  The net effect of the definition of \( f_s \) is that if \( t \) is the last stage where some ball may have been placed at \( \alpha\concat[0] \) then \( f_s(\alpha)\not=0 \)  unless \( p_\alpha(t)\conv[s] \).  Thus, the ball is not released until the delay demanded by \( \mathcal{N}_\alpha \) after entering \( X^2_{\alpha_2}  \) has expired.  If \( \mathcal{N}_\alpha \) is reset, we reset the sets $X^2_{\alpha_1}$ and $X^2_{\alpha_2}$.   Specifically, at the next stage \( \alpha \) is visited, we initialize  \( X^2_{\alpha_1} = \set{x < s}{\curnode{{x}}{{s}} \nsubseteq \alpha } \) and \(   X^2_{\alpha_2}=\eset \) before continuing as usual.

\subsubsection{The \texorpdfstring{\( \mathcal{P}_\alpha \)}{P_\alpha} module}

If \( \mathcal{Q}_\alpha \) implements \( \mathcal{P}_\alpha \) then we define

\begin{equation*}
f(\alpha) = \begin{cases}
																						0 & \text{if } \REset{\alpha} \isect A  \neq \eset \\	
																						1 & \text{if } \REset{\alpha} \subseteq \overline{A} \text{ and finite}
																					\end{cases}
\end{equation*}

\begin{equation*}
f_{s}(\alpha) = \begin{cases}
																						0 & \text{if } A_s \isect \REset[][s]{\alpha} \neq \eset 
\\																						1 & \text{otherwise}
									\end{cases}
\end{equation*}

The action of \( \mathcal{P}_\alpha \) tries to place some element from \( \REset{\alpha} \) into \( A \).  When there is an \( x \) available for \( \alpha \) already in \( \REset[][s]{\alpha} \),  requirement \( \mathcal{P}_\alpha \) places \( x \) on the largest node \( \beta\, \concat\, 0 \subset \alpha \) with \( \mathcal{Q}_\beta \) implementing a negative requirement.

\subsubsection{The \texorpdfstring{\( \mathcal{R}_\alpha \)}{R_\alpha} module}

If \( \mathcal{Q}_\alpha \) implements \( \mathcal{R}_\alpha \), then we define

\begin{equation*}
f(\alpha) = \begin{cases}
																						0 & \text{if }  \limsup_{s} \card{\REset(A)[\Lstg{\alpha}{s}]{e}} = \infty \\	
																						1 & \text{if }  \limsup_{s} \card{\REset(A)[\Lstg{\alpha}{s}]{e}} < \infty 
						\end{cases}	
\end{equation*}
	
\begin{equation*}	l^{\alpha}(s)=\card{\REset(A_{t})[t]{e}} \text{ where } t=\Lstg{\alpha}{s}
\end{equation*}
For $n\le l^{\alpha}(s)$,
\begin{equation*}
		u^{\alpha}_s(n)=\murec{t\leq s}{\exists x_1\not= x_2\not=\ldots \not=x_n}\forall[j \leq n]\forall[{t' \in [t,s]}] [x_j \in \REset(A_{t'})[t']{\alpha} ] 
\end{equation*}
																					
\begin{equation*}
f_{s+1}(\alpha) = \begin{cases}
																						
																						0 & \text{if } l^{\alpha}(s+1) > l^{\alpha}(s) \\
																						1 & \text{otherwise}
									\end{cases}
\end{equation*}

The role of \( \mathcal{R}_\alpha \) is to define a restraint function that preserves computations of the form  \( \card{\REset(A)[s]{\alpha}}=n \).  To this end, if \( \alpha \) is visited at \( s \) then set

\begin{equation*}
	\rstrn[\alpha]{\beta}{s} = \begin{cases}
																	0 & \text{if } \beta \nsupseteq \alpha\concat\, 0 \text{, otherwise }\\
																	s & \text{if } \lh{\beta}-\lh{\alpha} > l^{\alpha}(s)\\
																	u^{\alpha}_s(\lh{\beta}-\lh{\alpha}) & \text{else } 
															\end{cases}
\end{equation*}
Notice that no restraint is placed on nodes \(\beta \nsupseteq \alpha\concat\, 0\).
Our convention about the use of \( \card{\REset(A_s)[s]{\alpha}} \) and the action of the tree guarantees that no \( \mathcal{P}_\beta \) for  \(\beta \nsupseteq \alpha\concat\, 0 \) is able to remove elements from \( \REset[A_s][s]{\alpha} \).  On the other hand, the restraint function ensures that there are only finitely many nodes that might disrupt the computations responsible for enumerating the \( n \) longest residing elements of \( \REset[A_s][s]{\alpha} \) for \(n\le l^{\alpha}(s)\).  We will demonstrate below that this ensures a variant of \( \mathcal{R}_\alpha \) is satisfied.

\subsection{Verification}

\subsubsection{Tree Requirements}

The construction above does not quite satisfy the requirements given above.  Instead, for \( \alpha \) along the true path, we show the following modified versions of the requirements are satisfied.

\begin{align*}
	\tag*{ $\mathcal{N}_{\alpha}$ } & p_\alpha \text{ is partial or  there is a 2-c.e.\ set }  X^2_\alpha =^{*} \overline{A} \text{ s.t. } X^2_{\alpha,s} \isect A_{p(s)} =\eset. \\
	\tag*{ $\mathcal{P}_{\alpha}$ } & \card{\REset{\alpha}} = \infty \implies A \isect \REset{\alpha} \neq \eset.\\
	\tag*{ $\mathcal{R}_{\alpha}$ } & \limsup_{s\to\infty} l^{\alpha}(s)= \infty \implies \card{\REset(A){\alpha}} = \infty.  
\end{align*}

The requirements \( \mathcal{N}_{\alpha} \) and \( \mathcal{P}_{\alpha} \) differ only in notation.  The requirement \( \mathcal{R}_{\alpha} \) differs in that   \( \limsup_{s} \card{\REset(A_s)[s]{\alpha}}  \) is only evaluated at \  stages \( s \) such that  \( \alpha \) is on the true path.  This avoids any transitory effects that might enumerate elements into \( \REset(A_s)[s]{\alpha} \) and then remove them again before \( \mathcal{R}_\alpha \) has a chance to act.  As we will see, however, it is no less effective a means to show that \( \jjump{A}=\zerojj \).

\subsubsection{Satisfaction}

The essential property satisfied by a \( \Pi^0_2 \) tree construction is:

\begin{lemma}\label{lem:lim-inf}
	\( f = \liminf_s f_s \)
\end{lemma}

Recall that we write \( f_s(\alpha) \) to denote \( f_s(\lh{\alpha}) \) provided \( f_s \supseteq \alpha \).  In our construction, Lemma \ref{lem:lim-inf} can be verified by straightforward inspection of the modules where we understand \( f(n) \) only to be defined if \( f_s(n) \) is defined for infinitely many \( s \).  We will later show that \( f \) is a total function but first we observe three important properties of our construction. 

\begin{lemma}\label{lem:visitation-and-motion}
\noindent

	\begin{enumerate}
		\item If \( \mathcal{Q}_\alpha \) is visited at stage \( s \), there are no balls at any \( \beta \subset \alpha \) not in \( A \).
		\item If \( f_s \supseteq \curnode{x}{s-1} \) then \( \curnode{x}{s} \subsetneq \curnode{x}{s-1} \), provided \(\curnode{x}{s-1} \) and \(\curnode{x}{s} \) are defined.
		\item If \( \alpha \) is on the true path, then there is a stage \( s \) after which \( \mathcal{Q}_\alpha \) is never reset.
	\end{enumerate}
\end{lemma}

The first two follow directly from the action of the tree, and the third property follows immediately from Lemma \ref{lem:lim-inf}.  We now show that  the true path is total.  We first note that no individual ball ever stalls on the true path.

\begin{lemma}\label{lem:motion-forward}
	If \( \curnode{x}{s}=\alpha \) and \( f_t \supseteq \alpha \) where \( t >s \), then either \( \mathcal{Q}_\alpha \) was reset between stage \( t \) and \( s \) or \( x \in A_t \). 
\end{lemma}
\begin{proof}
	The only way \( x \) can leave \( \alpha \) is if \( \alpha \) is reset or \( x \) moves to a predecessor of \( \alpha \).  By induction, \( x \) must either reach \( A \) or \( \alpha \) is reset when its predecessor is reset. 
\end{proof}

We need to know that no positive requirement emits so many balls that \( f_s \) cannot extend a particular node.

\begin{lemma}\label{lem:positive-finite}
	If \( \alpha \subseteq f \) then \( \mathcal{P}_\alpha \) places at most finitely many balls on the tree.  Furthermore, if \( \alpha \subseteq f \) and \( \beta \supset \alpha \), then only finitely many balls placed on the tree by \( \mathcal{P}_\beta \) travel down the tree to reach \( \alpha \).
\end{lemma}
\begin{proof}
	If \( \alpha \) places infinitely many balls on the tree, then there is a stage \( s \)  and a ball \( x  \in \REset[s]{\alpha} \) so that \( \mathcal{P}_\alpha \) emits \( x \) at stage \( s \) and \( \alpha \) is no longer reset after stage \(s\).  Furthermore, \( \mathcal{P}_\alpha \) must emit another ball later, so there is a stage \( t > s \) with \( f_t \supseteq \alpha \). Hence,  \( x \in A_t \) by \ref{lem:motion-forward}.  By the action of the \( \mathcal{P}_\alpha  \) module,  however, \( \mathcal{P}_\alpha  \) ceases emitting once \( x \in A_t \isect \REset[t]{\alpha}\not=\eset \). Thus,  only finitely many balls are placed on the tree by \( \mathcal{P}_\alpha \).  The second half of the claim follows by the same argument applied to the stages at which \( x \) reaches \( \alpha \). 
\end{proof}

\begin{lemma}\label{lem:total-function}
	The function \( f \) is total, and every \( \alpha \subset f \) is visited infinitely many times.
\end{lemma}
\begin{proof}
	Suppose not; then \( f=\alpha \) for some \( \alpha \in \bstrs \) and \( \set{s}{f_s=\alpha} \) is infinite.  Without loss of generality, we may assume that \( f_s \) never properly extends \( \alpha \) for \( s > 0 \). By construction, this occurs for large \( s \) only  if every time \( f_s=\alpha \), there is some \( x \) with \( \curnode{x}{s} =\alpha^{-} \).  Pick \( s_0 \) large enough so that \( \alpha \) is never reset after \( s_0 \), and no \( \mathcal{P}_\beta \) with \( \beta \subseteq \alpha \) places any ball on the tree after \( s_0 \).  Now, pick some \(t\) and \(s_1\) where \( t> s_1 > s_0 \) such that \( f_t=f_{s_1} = \alpha \) and no nodes above \(\alpha\) are visited between stages $t$ and $s_1$.     If \( x \) is such that \( \curnode{x}{t}=\alpha^{-} \),  the ball \(x\) cannot have trickled down from above  nor can it have been placed at \( \alpha^{-} \) after \( s_0 \).  Hence, \( \curnode{x}{s_1} = \alpha^{-} \), violating Lemma \ref{lem:motion-forward}, a contradiction.
\end{proof}

Before we can conclude that \( \mathcal{P}_\alpha \) is satisfied, we first must argue that \( \mathcal{R}_\alpha \)  imposes only  finitary restraint on the true path.  We need two further lemmas.  The first  shows that the only way ball \( x \) can pass by ball \( y \) is if \( y  \) is placed on the tree first and later \( x \) is added with \( y <_L x \).

\begin{lemma}\label{lem:free-above}
	Suppose \( x \) is placed on the tree at stage \( s \) and 
	\( \curnode{y}{s} \not\supset\curnode{x}{s} \).  If \( x \) remains on the tree until \( s' > s \) and \mbox{\( \curnode{y}{s'} \supseteq \curnode{x}{s'} \),} then \(  \curnode{y}{s} <_L \curnode{x}{s} \). 
\end{lemma}
\begin{proof}
	Since the construction guarantees that the nodes above the position of \( x \)  are not visited, \( y \) cannot be added to the tree above \( x \).  Since balls move downward on the tree, if \( t+1 > s \) is the least stage at which \( \curnode{y}{t+1} \supseteq \curnode{x}{t+1} \), we have either \( \curnode{x}{t} <_L \curnode{y}{t} \) or \( \curnode{y}{t} <_L \curnode{x}{t}  \).  However, the motion of \( x \) at stage \( t \) requires that \( f_{t+1} \supseteq \curnode{x}{t} \). So, if \( \curnode{x}{t} <_L \curnode{y}{t} \), then \( y \) is removed from the tree at stage \( t+1 \), contradicting our assumption that \( \curnode{y}{t+1} \supseteq \curnode{x}{t+1} \).  Therefore, \( \curnode{y}{t} <_L \curnode{x}{t}  \).  Provided that the balls remain on the tree between \( s \) and \( t \), we have \( \curnode{y}{t} \subseteq \curnode{y}{s} \) and \( \curnode{x}{s} \supseteq \curnode{x}{t} \).  If \( y \) was added to the tree to the left of the location of \( x \) after \(x\) was placed on the tree, then \( x \) would have been removed from the tree.  So, \( \curnode{y}{s} <_L \curnode{x}{s}\).
\end{proof}

We can now infer that if \( \alpha \) is on the true path, then infinitely often the only balls above \( \alpha \) are those that will never move below \( \alpha \) without being reset.  Let \[
	B_t = \set{x}{\curnode{x}{t} \supseteq \alpha \land \exists[t'>t]\left(\curnode{x}{t'}\subseteq \alpha\right)}.
	\]
Let \( \hat{B_t}  \) be the set of elements in \( B_t \) that reach \( \alpha \) without being reset between stages $t$ and $t'$ in the definition of \(B_t\).

\begin{lemma}\label{lem:none-above}
	If \( \alpha \subset f \) then for every \( s \)  there is a \( t \ge s\) such that \(f_t\supset\alpha\) and \( \hat{B_t}=\eset.  \)  
\end{lemma}
\begin{proof}
	Without loss of generality, we may assume \( s \) is so large that \( \alpha \) is never reset after \( s \) and that every positive requirement below \( \alpha \) no longer acts.    Suppose \( \hat{B_s}\not =\eset.  \) Let \( x \in \hat{B_s} \) be such that \(\curnode{x}{s}\) is  minimal in \( \hat{B_s} \) under \( <_L \) and maximal among those elements under \( \supset \).  By definition of \( B_s \), there is a least stage \( t>s \) at which \( \curnode{x}{t} \subseteq \alpha \).  Now, given any \( y \in B_t \), we have \( \curnode{y}{t} \supsetneq \alpha \supset \curnode{x}{t} \), and by maximality of \( x \) under \( \subset \), we know that \( \curnode{y}{s} \not\supset \curnode{x}{s} \). So, by Lemma \ref{lem:free-above}, we know that \(  \curnode{y}{s} <_L \curnode{x}{s} \).  Since neither \( x \) nor \( y \) was reset between stages \( s \) and  \( t \), the element \( y \) is an element in \( \hat{B_s} \) to the left of \( x \), a contradiction.  Hence, \( \hat{B_t}=\eset \). The second part of the claim follows by choosing \( t' \geq t \) to be the stage at which \( x \) enters \( A \) and again applying the previous lemma to show nothing from the right could get above \( x \).
\end{proof}

Recall 	\(\rstrn{\beta}{s} = \max_{\alpha \subset \beta} \rstrn[\alpha]{\beta}{s} \) is the restraint imposed on \(\mathcal{P}_\beta\) by \( \mathcal{R}_\beta \).
\begin{lemma}\label{lem:finite-restraint}
	If \( \alpha \) is on the true path, then \( \mathcal{R}_\alpha \) is satisfied, and if \( \beta \supset \alpha \) is on the true path,  then \( \lim_{s \to\infty} \rstrn[\alpha]{\beta}{s} \) is finite.  Hence,  \( \lim_{s \to\infty} \rstrn{\beta}{s} \) is finite as well.  

\end{lemma}
\begin{proof}
If \( \limsup_{s} l^{\alpha}(s)  < \infty \), then \( \mathcal{R}_\alpha \) is satisfied, and the second half follows trivially (since \(\beta \nsupseteq \alpha\, \concat\ 0\)).  Otherwise, we work after a stage large enough such that \( \alpha \) is not reset anymore and no positive requirements below \( \alpha \) place balls on the tree. We claim that for every \( n \) there is some stage \( s_n \) and elements \( x_1, x_2, \ldots x_n \) satisfying

\begin{enumerate}
\item \(f_{s_n} \supseteq \alpha \),
\item \(\forall[i\neq j\le n] [x_{i} \neq x_j ] \), and 
\item \(\forall[i\le n]\forall[t\geq s_n]  [x_i \in \REset(A_t)[t]{\alpha} ]\).
\end{enumerate}

By the definition of \( u^{\alpha}_s(n) \), the last equation entails that \mbox{\( s_n \geq u^{\alpha}_t(n) \)} since  element \( x_n\in \REset(A_t)[t]{\alpha} \) for every \( t \geq s_n \).

To verify the claim,  suppose \( n \) is the least failure of this claim.  Pick \( s > s_{n-1} \) large enough such  that every ball placed on the machine by \( \mathcal{P}_\beta \) with \( \beta \supset \alpha \) and \( \lh{\beta} - \lh{\alpha} < n \) which will ever enter \( A \) has already done so and at which \(  l^{\alpha}(s)  < n \) .  Now pick \( s' > s \) as given in Lemma \ref{lem:none-above} and \( t \ge s'  \) least with \( f_{t} \supseteq \alpha\concat\, 0 \).  Note that \( \hat{B_t}=\eset  \) as well.   At stage \( t \), the only balls above \( \alpha \) are those added at this very stage and those that will be reset before they get below \( \alpha \).

By the strategy given for \( \mathcal{R}_\alpha \), we know that \( l^{\alpha}(t) > l^{\alpha}(t-1) \geq n-1 \).  Let  \( x_n \) be the element that has occupied \( \REset(A_t)[t]{\alpha} \) for the longest uninterrupted time and is not equal to any of \( x_{m}\) for \(m < n \).  Since \( f_t <_L \alpha\concat[1] \), no \( \mathcal{P}_\beta \) with \( \beta \supseteq \alpha\concat[1]  \) can add balls less than \( t \) to \( A \) and, thus, cannot remove \( x_n \) from \( \REset(A){\alpha} \).  On the other hand, if \( \beta \supseteq \alpha\concat[0] \) and if \( \mathcal{P}_\beta \) is to succeed in placing any elements in \( A \), we must have \( \lh{\beta} - \lh{\alpha} \geq n \).  This  implies that, for all \( t' \geq t \), if \( x_n \in \REset(A_{t'})[t']{\alpha} \) and \( \hat{s} \) is the first stage at which \( x_n \) entered \( \REset(A){\alpha} \) and remained in until \( t' \), then \(\rstrn[\alpha]{\beta}{t'}\geq\hat{s} \).  Thus, the restraint guarantees any new balls placed on the tree cannot remove \( x_n \) from \( \REset(A){\alpha} \).  Moreover, any old balls already on the tree above \(\alpha\) are reset before they get to \(\alpha\) because  \( \hat{B_t}=\eset  \), so they cannot remove \( x_n \) from \( \REset(A){\alpha} \).  This verifies the claim, and the lemma follows immediately.
\end{proof}

\begin{lemma}\label{lem:negative-met}
\( A \) is \( 2 \)-tardy and simple.
\end{lemma}
\begin{proof} 
The sets built at the \( \mathcal{N}_\alpha \) nodes along the true path witness that \( A \) is \( 2 \)-tardy.   Suppose $p_\alpha$ is total.  Let $s'$ be the last stage at which $\alpha$ is reset.  For each stage $s\ge s'$ such that  $f_s\supset\alpha$, all elements $x <s$ with $x\not\in A_s$ for which  $\curnode{x}{s}\not\subseteq\alpha$  are enumerated into $X^2_{\alpha_1}$.  By construction, any ball in $X^2_{\alpha_1}$ that passes through node $\alpha$ after stage $s'$ enters $X^2_{\alpha_2}$ and is appropriately delayed by $p_\alpha$ before entering $A$.  Consider a ball $y$ on a node $\beta'\supseteq\alpha'$ for $\alpha'\subset\alpha$ and $\beta'>_R \alpha$ at some stage $t>s'$ where $\beta\subset f_t$.  Let $t'$ be the least stage greater than $t$ such that  $f_{t'}\supset\alpha$.  If  $y$ is on a node to the right of $\alpha$ at stage $t'$, the ball $y$ is recycled (and hence is not added to  $X^2_{\alpha_1}$ at this stage).  Otherwise,  $y$  has been reset or has already entered $A$ by stage $t'$ via the node $\alpha'$ without being placed in $X^2_{\alpha_1}$ or $X^2_{\alpha_2}$. 
  By Lemma \ref{lem:finite-restraint}, if \( \REset{\beta} \) is an infinite c.e.\ set, then eventually some element in \( \REset{\beta} \) is greater than the finite value \( \lim_{s \to\infty} r(\beta, s) \) 
 and enters \( \REset{\beta} \) after the last stage at which \( \beta \) is  reset.  Thus, \( \mathcal{P}_\beta \) will succeed in making \( A \isect  \REset{\beta}  \neq \eset \).
\end{proof}

\begin{lemma}\label{lem:low2}
	\( \jjump{A} \Tleq \zerojj \)
\end{lemma}
\begin{proof}

 Recall that  requirement \( \mathcal{R}_\alpha \) guarantees that \(  \REset(A){\lh{\alpha}} \) is infinite if and only if \mbox{\( \limsup_{s} l^{\alpha}(s)= \infty  \).} 
 To determine whether \( \REset(A){\lh{\alpha}} \) is infinite, \( \zerojj \) simply computes \( f(\alpha) \).
 Thus, 	\( \jjump{A} \Tleq \zerojj \).
	
\end{proof}

This completes the proof of the theorem.

\section{Open Questions}

First, we would like to know whether there are any $n$-tardy sets that are not automorphic to a complete set that do not satisfy $Q_n$. 

\begin{question}
  Does every $n$-tardy set not automorphic to a complete set satisfy
  $Q_n$?
\end{question}

We would also like to to know whether there are properly very tardy sets that are not automorphic to a complete set.  

\begin{question}
 Is there a very tardy set that is not $n$-tardy for any $n$ and is not automorphic to a
  complete set?
\end{question}

The above question could be attacked using definable properties.  We have not yet found a property that describes the properly very tardy sets, i.e., those very tardy sets that are not $n$-tardy for any $n\in\omega$.  

\begin{question}
  Find a property ${Q}_\infty$ so that if
  ${Q}_\infty(A)$ holds, then $A$ is very tardy, and find some very tardy set \( A \) that is not \( n \)-tardy for any \( n \) and satisfies \( Q_\infty(A) \).  
\end{question}

Finally, we want to know whether Theorem \ref{T:3tardy} can be extended as follows.

\begin{question}
Is there a properly  $n+1$-tardy set that is not computed by any $n$-tardy sets?
\end{question}

\newcommand{\MR}[1]{\MRhref{#1}{MR #1}}
\newcommand{\MRhref}[2]{\href{http://www.ams.org/mathscinet-getitem?mr=#1}{#2}}


\end{document}